\documentclass[11pt]{amsart} \textwidth=14.5cm \oddsidemargin=1cm
\usepackage{amsmath,wasysym}
\usepackage{amsxtra}
\usepackage{amscd}
\usepackage{amsthm}
\usepackage{amsfonts}
\usepackage{amssymb}
\usepackage{eucal}
\usepackage{graphics, color}
\usepackage{hyperref,mathrsfs}
\usepackage[usenames,dvipsnames]{xcolor}
\usepackage{tikz}
\usepackage{stmaryrd}
\usepackage[all]{xy}
\usepackage{hyperref}
\usepackage{color}
\usepackage{bbm} 
\allowdisplaybreaks

\hypersetup{colorlinks,linkcolor=blue, urlcolor=blue,
citecolor=blue}

\newtheorem{thm}{Theorem} [section]
\newtheorem{lem}[thm]{Lemma}
\newtheorem{cor}[thm]{Corollary}
\newtheorem{prop}[thm]{Proposition}

\theoremstyle{definition}
\newtheorem{definition}[thm]{Definition}

\newtheorem{rem}[thm]{Remark}

\numberwithin{equation}{section}

\newcommand{\A}{\mathcal A}

\newcommand{\B}{{\bf B}}

\newcommand{\C}{{\mathbb C}}
\newcommand{\co}{\text{co}}
\newcommand{\D}{\mathscr{D}}

\newcommand{\End}{{\mathrm{End}}}

\newcommand{\ff}{\mathtt{f}}
\newcommand{\sF}{{\mathscr F}}

\newcommand{\HH}{\mathbf H}
\newcommand{\Hom}{{\mathrm{Hom}}}
\newcommand{\id}{\mathbbm{1}}
\newcommand{\Iw}{\mathfrak{I}}

\newcommand{\mi}{\mathbf{i}}
\newcommand{\mj}{\mathbf{j}}
\newcommand{\mk}{\mathbf{k}}

\newcommand{\N}{{\mathbb N}}
\newcommand{\Ob}{\mathscr O}

\newcommand{\Q}{\mathbb {Q}}

\newcommand{\ro}{\text{ro}}
\newcommand{\Sc}{\mathbf S}
\newcommand{\Scg}{{\mathbf S}_v'} 
\newcommand{\T}{\mathbf T}

\newcommand{\Waff}{W_{\mathrm{aff}}}
\newcommand{\Z}{{\mathbb Z}}

\newcommand{\nc}{\newcommand}
\nc{\browntext}[1]{\textcolor{brown}{#1}}
\nc{\greentext}[1]{\textcolor{green}{#1}}
\nc{\redtext}[1]{\textcolor{red}{#1}}
\nc{\bluetext}[1]{\textcolor{blue}{#1}}
\nc{\brown}[1]{\browntext{ #1}}
\nc{\green}[1]{\greentext{ #1}}
\nc{\red}[1]{\redtext{ #1}}
\nc{\blue}[1]{\bluetext{ #1}}

\title[Cells in affine $q$-Schur algebras]
{Cells in affine $q$-Schur algebras}

\author[Weideng Cui]{Weideng Cui}
\address{School of Mathematics, Shandong University, Jinan, Shandong 250100, China}
\email{cwdeng@amss.ac.cn (Cui)}

\author[Li Luo]{Li Luo}
\address{School of Mathematical Sciences,
Shanghai Key Laboratory of Pure Mathematics and Mathematical Practice,
    East China Normal University, Shanghai 200241, China}
\email{lluo@math.ecnu.edu.cn (Luo)}

\author[Weiqiang Wang]{Weiqiang Wang}
\address{Department of Mathematics\\ University of Virginia\\ Charlottesville, VA 22904}
\email{ww9c@virginia.edu (Wang)}

\begin{document}

\begin{abstract}
We develop algebraic and geometrical approaches toward canonical bases for affine $q$-Schur algebras of arbitrary type introduced in this paper.  A duality between an affine $q$-Schur algebra and a corresponding affine Hecke algebra is established. We introduce an inner product on the affine $q$-Schur algebra, with respect to which the canonical basis is shown to be positive and almost orthonormal. We then formulate the cells and asymptotic forms for affine $q$-Schur algebras, and develop their basic properties analogous to the cells and asymptotic forms for affine Hecke algebras established by Lusztig. The results on cells and asymptotic algebras are also valid for $q$-Schur algebras of arbitrary finite type.
\end{abstract}

\maketitle
\setcounter{tocdepth}{1}
\tableofcontents

\section{Introduction}

\subsection{Background}

The $q$-Schur algebras of finite type A were introduced in \cite{DJ89, BLM90} via algebraic and geometric approaches. Recently, motivated by connections to BGG category $\mathcal O$, two of the authors \cite{LW22} formulated $q$-Schur algebras $\Sc^{\text{fin}}$ of arbitrary finite type, associated to an arbitrary subset $\check X_\ff$ of (co)weights invariant under a finite Weyl group $W$. For some specific choices of $\check X_\ff$ in type A, this version of $q$-Schur algebras reduces to the constructions in  \cite{DJ89, BLM90}, and for other specific choices of $\check X_\ff$ in type B, it recovers the $q$-Schur algebras in \cite{Gr97, BKLW18, FL15}; this version of $q$-Schur algebras is different from other constructions in the literature; for example, the Hecke endomorphism algebras studied in  \cite{DJM98, DS00} contain various $q$-Schur algebras $\Sc^{\text{fin}}$ of type B/C/D as proper subalgebras. The canonical basis for $\Sc^{\text{fin}}$ was constructed, and a geometric realization of $\Sc^{\text{fin}}$ and its canonical basis with positivity was obtained in \cite{LW22}.

The notion of left, right, and two-sided cells is natural from the viewpoint of canonical basis of Hecke algebras \cite{KL79}. A theory of cells for Hecke algebras has been developed systematically by Lusztig \cite{Lu85b, Lu87a, Lu87b, Lu89a, Lu95, Lu03} (see also \cite{Shi86, Xi94, BFO09}), and it has applications and connections to primitive ideals, nilpotent orbits, representations of finite groups of Lie type. Some of the main properties for cells of Hecke algebras are summarized in \cite{Lu03} and become known as Properties (P1)--(P15); these properties have been established for finite/affine Hecke algebras of equal parameters among others, and some results require the positivity of canonical bases, which are now known for Hecke algebras of general Coxeter groups thanks to Elias-Williamson \cite{EW14}.

The properties of cells in Hecke algebras of finite and affine type A are much better understood (see \cite{Lu85a, Shi86}), and  in finite type A it provides a reinterpretation of the classic Robinson-Schensted correspondence; however there are several unsettled conjectures for cells in affine Hecke algebras beyond type A (see \cite{Lu89a}).

On the other hand, the $q$-Schur algebras of affine type A has been studied in \cite{GV93, Gr99, Lu99, Lu00, Mc12}; also see \cite{FLLLWa, FLLLWb} for affine type C.

\subsection{Goal}

The main goal of this paper is to develop systematically a theory of cells for the $q$-Schur algebras of finite and affine types in wide generality. To that end, we will first formulate the notion of $q$-Schur algebras of arbitrary affine type  in the generality as done in \cite{LW22} for finite type. We then formulate and establish the Schur algebra counterparts of various basic concepts and properties of cells for Hecke algebras such as $\mathfrak a$-function, Properties (P1)--(P15), and asymptotic algebras.

A theory of canonical basis for $\imath$quantum groups arising from quantum symmetric pairs have been developed in \cite{BW18a, BW18b} (and in \cite{BKLW18, LiW18, FLLLWa} for finite/affine type AIII). Therefore it is natural to expect that a theory of cells for $\imath$quantum groups should exist (for cells in quantum groups, see \cite{Lu95} and  also \cite{Mc03}). The results of this paper in {\em classical} types will be used in a subsequent work to develop cells for $\imath$quantum groups of (finite/affine) type AIII. Since the notion of cells makes sense for $q$-Schur algebras of arbitrary type, it seems reasonable to formulate and study cells at such a generality first.

\subsection{Cells}

The notion of $q$-Schur algebras of arbitrary affine type starts with a subset $\check X_\ff$ of coweights invariant under an (extended) affine Weyl group $\widetilde W$. Our construction when taking $\check X_\ff=\check X$ in affine type A reduces to the aforementioned construction studied in  \cite{GV93, Gr99, Lu99, Lu00}. It is rather straightforward to formulate and establish the affine counterparts of the main results on the canonical bases, positivity, and their geometric realizations in \cite{LW22} for finite type $q$-Schur algebras.

Based on canonical bases, cells and asymptotic forms for $q$-Schur algebras of finite type A (\`a la \cite{BLM90, DJ89}) have been formulated by Du \cite{Du95, Du96}, while cells and asymptotic forms for $q$-Schur algebras of affine type A have been studied by McGerty \cite{Mc03}. This is made possible largely thanks to the $q$-Schur duality and the simpler structures of cells for Hecke algebras of finite and affine type A.

The formulations of cells for the finite and affine $q$-Schur algebras of type A involve various formulas in terms of matrices in \cite{Du92, Du95, DDPW, Mc03} (for type B/C, see \cite{BKLW18, FLLLWa, FL15}), which originated in \cite{BLM90, Lu99}. In the generality of affine $q$-Schur algebras of this paper, these formulas are often replaced uniformly by formulas involving affine Weyl groups.

In this paper, we formulate an inner product on affine $q$-Schur algebras geometrically. We establish a precise connection between this inner product and the inner product defined algebraically in \cite{Wi08, Wi11} for Schur algebroids. In particular, this establishes the positivity and almost orthonormality of the inner product with respect to canonical basis. Our general approach is actually simpler than some other approaches for affine type A and C which cannot be generalized to our setting (see \cite{Mc12, LiW18}).

We are able to formulate and establish Properties (P1)--(P15) (with one exception which seems less relevant in the $q$-Schur algebra setting) for affine $q$-Schur algebras as well as asymptotic Schur algebras. Under a mild regularity assumption, we show that there is a one-to-one correspondence between the two-sided cells in the affine $q$-Schur algebra and the two-sided cells in the corresponding extended affine Hecke algebra; see Corollary~\ref{lem:d=d3}.

It will also be interesting to develop a theory of cells for $q$-Schur algebras (of both finite and affine type in wide generality) of unequal parameters (cf. \cite{Lu03} for Hecke algebras of unequal parameters).

\subsection{Organization}

The paper is organized as follows.
In Sections~\ref{sec:prelim}--\ref{sec:geom} we develop the basic constructions and properties for affine $q$-Schur algebras in wide generality, following \cite{LW22}. This also allows us to set up notations which are to be used in the rest of the paper.

In Section~\ref{sec:inner}, we fomulate an inner product on an affine $q$-Schur algebra, and show that the canonical basis is positive and almost orthonormal with respect to the inner product; see Theorem~\ref{thm:inner}.

In Section~\ref{sec:cells}, a characterization on the left, right, and two-sided cells for affine $q$-Schur algebras is provided; see Proposition~\ref{prop:d=d2}. Properties (P1)--(P15) for affine $q$-Schur algebras are formulated in Theorem~\ref{thm:cell structures}.

In Section~\ref{sec:asymp}, we formulate the asymptotic Schur algebras, and establish their relations with affine $q$-Schur algebras and their left, right, and two-sided cells; see Proposition~\ref{characterization-two-cells}, Theorem~\ref{prop:d=d2=a'aa}, Proposition~\ref{prop:double-centralizer-property}.

We remark that the main results in Sections~\ref{sec:inner}--\ref{sec:asymp} are also new and valid for $q$-Schur algebras of finite type with the same proofs.

\vspace{2mm}
\noindent {\bf Acknowledgement.}
We thank Geordie Williamson for helpful discussions on the inner product for Schur algebroids. WC is partially supported by China Scholarship Council, Young Scholars Program of Shandong University and the NSF of China (grant No. 11601273). WC thanks University of Virginia for hospitality and support. LL is partially supported by the Science and Technology Commission of Shanghai Municipality (grant No. 18dz2271000, 21ZR1420000) and the NSF of China (grant No. 11871214). WW is partially supported by the NSF grant DMS-1702254 and DMS-2001351.

\section{Affine $q$-Schur algebras of arbitrary type}
\label{sec:prelim}
\subsection{Affine Weyl groups and associated Hecke algebras}

Let $\Phi$ be the root system and $\check{X}$ be the coweight lattice. Fix a simple system $\Pi=\{\alpha_1,\ldots,\alpha_d\}$ and denote by $\{\alpha^\vee_1,\ldots,\alpha^\vee_d\}$ their coroots. We denote by
\[
\check{X}^{\rm dom} =\{\mi \in \check X \mid \langle \alpha_k, \mi \rangle \in \N, \forall 1\le k \le d\}
\]
the set of dominant coweights.

The (finite) Weyl group $W$ is generated by the simple reflections $\{s_1,s_2,\ldots,s_d\}$ with identity $\id$.
The \emph{extended affine Weyl group} $\widetilde{W} =\check{X} \rtimes W$ is
\begin{equation*}
\widetilde{W}=\{t_\mi w~|~ \mi\in \check{X}, w\in W\}\quad\quad\mbox{with}\quad\quad t_\mi t_\mj=t_{\mi+\mj}\quad \mbox{and}\quad wt_\mi=t_{w(\mi)}w.
\end{equation*}
Let $\theta \in \Phi$ be the unique highest root and $\theta^\vee= \frac{2 \theta}{(\theta,\theta)}$ its coroot. Set $s_0=t_{\theta^\vee} s_{\theta}$; we have $s_0^2=1$. The \emph{(non-extended) affine Weyl group} $\Waff$ is the subgroup of $\widetilde{W}$ generated by $s_0,s_1,\ldots,s_d$.
There exists (cf. \cite{Lu87b}) a finite abelian group $\Omega$ of $\widetilde{W}$ such that
\begin{equation*}
\widetilde{W}=\Omega\ltimes\Waff,
\end{equation*}
and a length function $\ell: \widetilde{W}\longrightarrow \mathbb{Z}$ with $\ell(s_i)=1$, for $0\le i \le d$, and
\begin{equation}
 \label{eq:Ome}
\Omega=\{w\in \widetilde{W}~|~\ell(w)=0\}.
\end{equation}

There is a natural right action of $W$ on the coweight lattice $\check{X}$ defined by sending
\begin{equation}\label{action:W}
\mj\in \check{X} \mapsto \mj w=w^{-1}(\mj).
\end{equation}

Fix $n\in\Z_{\ge 2}$. There exists a right action of $\widetilde{W}$ on $\check{X}$ given by \eqref{action:W} and \eqref{action:t} below:
 \begin{equation}\label{action:t}
\mj t_\mi=\mj-n\mi.
\end{equation}
In particular, the right action of $s_0$ on $\check{X}$ is
\[
\mj s_0=s_{\theta}(\mj-n\theta^\vee).
\]

%

Let $v$ be an indeterminate and let
\[
\A=\Z[v,v^{-1}].
\]
The \emph{(extended) affine Hecke algebra} $\widetilde{\HH}$ of $\widetilde{W}$ is an $\A$-algebra which is a free $\A$-module with basis $H_w$, for $w\in \widetilde{W}$, such that
\begin{align*}
H_{w} H_{w'} = H_{ww'}, &\quad \mbox{if } \ell(ww')=\ell(w)+\ell(w');\\
(H_{s_i}-v^{-1})(H_{s_i}+v)=0, &\quad\mbox{for } i=0,1,2,\ldots,d.
\end{align*}
We shall write $H_i =H_{s_i}$ and $H_{\mi} =H_{t_\mi}$, for $\mi \in \check{X}$. The subalgebra $\HH$ of $\widetilde{\HH}$ generated by $H_1,H_2,\ldots,H_d$ is the Hecke algebra associated to $W$,  and the (non-extended) affine Hecke algebra $\HH_{\mathrm{aff}}$ of $\widetilde{\HH}$ is the subalgebra generated by $H_0,H_1,\ldots,H_d$.

For any $\mi\in\check{X}$, we define an element $\theta_\mi\in\widetilde{\HH}$ as follows. We write $\mi=\mj-\mk$
with $\mj,\mk\in \check{X}^{\rm dom}$. Set $\theta_\mi=H_{\mj} H^{-1}_{\mk}$, which is independent of the choice of $\mj$ and $\mk$. It has been shown in \cite{Lu89b} that
\begin{equation*}
\theta_{\mi}\theta_{\mj}=\theta_{\mj}\theta_{\mi},\quad (\forall\mi,\mj\in\check{X}).
\end{equation*}
Moreover, the elements $H_w\theta_{\mi}\ (w\in W, \mi\in\check{X})$ form an $\A$-basis for $\widetilde{\HH}$ (see \cite[Proposition~ 3.7]{Lu89b}), and there is an explicit commutation relation between $\theta_\mi$ and $H_k$ (see \cite[Proposition~3.6]{Lu89b}).

%
%
\subsection{Affine $q$-Schur algebras}

Fix an arbitrary $\widetilde{W}$-invariant subset of $\check{X}$:
\begin{equation*}
\check{X}_{\ff}\subset \check{X}.
\end{equation*}

Denote
\begin{equation*}
F=
\Big\{
\mi\in \check{X}~\big |~-n<\frac{2(\mi,\alpha)}{(\alpha,\alpha)}\leq 0, \forall \alpha\in\Phi^+
\Big \}.
\end{equation*}
Then $F\cap \check{X}_{\ff}$ is a fundamental domain for $\widetilde{W}$ acting on $\check{X}_{\ff}$ (cf. \cite[pp.233]{Jan03}).

We define the following $W$-invariant finite set:
\begin{align*}
F_{\ff}=W(F\cap \check{X}_{\ff}).
\end{align*}
Set
\begin{align}
\Lambda_{\ff}=\{\mbox{$\widetilde{W}$-orbits in $\check{X}_{\ff}$}\}.\label{lambdaf}
\end{align}
There exists a unique anti-dominant element in each $W$-orbit $\gamma\subset F_{\ff}$, which we denote by $\mi_\gamma$. Each $W$-orbit $\gamma\subset F_{\ff}$ determines a unique $\widetilde{W}$-orbit (also denoted by $\gamma$) in $\check{X}_{\ff}$. Thus we have a bijection
$$\Lambda_{\ff}\leftrightarrow \{\mbox{anti-dominant elements in $F_{\ff}$}\}, \quad \gamma \mapsto \mi_\gamma.$$

\begin{rem}
Note that the sets $F$ and hence $\Lambda_{\ff}$ are always finite. A standard choice of $\check{X}_{\ff}$ is $\check{X}_{\ff} =\check{X}$. \end{rem}

Associated to $F_{\ff}$ and $\check{X}_{\ff}$, we introduce the following free $\A$-modules
\begin{equation*}
\T =\bigoplus_{\mi\in \check{X}} \A u_\mi,
\qquad
 \T_{F_{\ff}} =\bigoplus_{\mi\in F_{\ff}} \A u_\mi,\qquad
 \T_{\ff}=\bigoplus_{\mi\in \check{X}_{\ff}} \A u_\mi,
\end{equation*}
with bases given by the symbols $u_\mi$, for $\mi\in \check{X}, F_{\ff}$ and $\check{X}_{\ff}$, respectively.
We shall refer to $\{u_\mi\}$ as the {\em standard basis} for $\T$, for $\T_{F_{\ff}}$ or for $\T_{\ff}$.

The right action of $\widetilde{W}$ on $\check{X}$ induces a right action of $\widetilde{W}$ on $\T$ (and $\T_{\ff}$) by
\begin{equation*}
u_{\mi}\cdot w=u_{\mi w}.
\end{equation*}
Define a right action of the Hecke algebra $\HH$ on $\T_{F_{\ff}}$ as follows:
\begin{align}\label{Haction}
u_{\mi}H_k&=
\left\{\begin{array}{ll}
v^{-1}u_{\mi}, & \mbox{if $\mi s_k=\mi$};\\
u_{\mi s_k}, & \mbox{if $\mi s_k\succ\mi$};\\
u_{\mi s_k}+(v^{-1}-v)u_{\mi}, & \mbox{if $\mi s_k\prec\mi$},
\end{array}\right. \quad\quad(\mi\in F_\ff, 1\leq k\leq d),
\end{align}
and define a right action of $\theta_\mj\ (\forall \mj\in \check{X})$ on $\T_{\ff}$ by
\begin{equation}\label{Taction}
u_\mi  \theta_\mj =u_{\mi-n\mj}, \quad (\forall \mi\in\check{X}_{\ff}).
\end{equation}
A right action of $\widetilde{\HH}$ on $\T_{\ff}$ is determined by \eqref{Haction}--\eqref{Taction}.

Introduce the following algebras (over $\A$ and $\Q(v)$, respectively):
\begin{equation}
  \label{eq:Schur}
\Sc_v=\Sc_v(\ff):=\End_{\widetilde{\HH}}(\T_{\ff}), \qquad \Sc_{v,\Q}=\Q(v)\otimes_{\A}\Sc_v.
\end{equation}
These algebras will be called {\em affine $q$-Schur algebras}. They include as special cases various affine $q$-Schur algebras in the literature where $\widetilde W$ is usually taken to be affine type A or C, cf. \cite{GV93, Lu99, Gr99, FLLLWa, FLLLWb}. Likewise, the $q$-Schur algebras of arbitrary finite type introduced in \cite{LW22} (as in \eqref{eq:Schur} with $\widetilde{\HH}$ replaced by $\HH$) include various $q$-Schur algebras in the literature as special cases (cf. \cite{BLM90, DJ89, Gr97, BKLW18}).

\subsection{A basis of affine $q$-Schur algebra}

For any {\em proper} subset $J\subset\{0,1,\ldots,d\}$, let $W_J$ be the parabolic (finite) subgroup of $\widetilde{W}$ generated by $\{s_j~|~j\in J\}$ and let $w_\circ^{J}$ be the unique longest element in $W_J$. Let $\D_J$ be the set of minimal length right coset representatives for $W_J\setminus \widetilde{W}$.
We define
\begin{align}
 \label{eq:xj}
 x_J=\sum_{w\in W_J}v^{\ell(w_\circ^{J})-\ell(w)}H_w.
\end{align}
(Our convention for $x_J$ here differs from some literature by a factor $v^{\ell(w_\circ^J)}$.)

For any $\gamma\in\Lambda_{\ff}$, we define the subset
\begin{equation}\label{Jgamma}
J_\gamma=\{k~|~0\leq k\leq d, \mi_\gamma s_k=\mi_\gamma\},
\end{equation}
Note that $J_\gamma$ is always a proper subset of $\{0,1,\ldots,d\}$ for any $\gamma\in\Lambda_{\ff}$.
Write
\begin{align*}
 w_\circ^{\gamma} =w_\circ^{J_\gamma},\quad
 W_\gamma =W_{J_{\gamma}},\quad
 \D_\gamma =\D_{J_\gamma},\quad
 x_\gamma=x_{J_\gamma}.
\end{align*}
The subspace of $\T_{\ff}$,
\[
\T_{\gamma}:=\bigoplus_{\mi\in\gamma}\A u_{\mi}
\]
is clearly a right $\widetilde{\HH}$-module.

\begin{lem}\label{identification} There is a right $\widetilde{\HH}$-module isomorphism
\[
\T_{\gamma}\cong x_\gamma\widetilde{\HH},
\qquad u_{\mi_\gamma}\mapsto x_\gamma,
\]
and hence,
$\T_{\ff}\cong\bigoplus_{\gamma\in\Lambda_{\ff}}x_\gamma\widetilde{\HH}.$
\end{lem}

\begin{proof}
As a right $\widetilde{\HH}$-module, $\T_{\gamma}$ (resp. $x_\gamma \widetilde{\HH}$) is generated by $u_{\mi_\gamma}$ (resp. $x_\gamma$). We need only show that no nonzero element of $x_\gamma \widetilde{\HH}$ annihilates the vector $u_{\mi_\gamma}$. Suppose
$$H=\sum_{w\in \D_\gamma}c_{w}x_\gamma H_w\neq0, \quad \mbox{($c_w\in \A$ are almost zero except finitely many $w$)}$$ kills $u_{\mi_\gamma}$.
Take $w_1 \in \D_{\gamma}$ such that $c_{w_1}\neq0$ and $\ell(w_1)$ is maximal. We see the coefficient of the term $u_{\mi_\gamma w_1}$ in
$u_{\mi_\gamma} H$ is nonzero, which is impossible.
\end{proof}

Thanks to the above lemma, we can identify
\begin{equation*}
\Sc_v=\End_{\widetilde{\HH}}(\oplus_{\gamma\in{\Lambda_{\ff}}}x_\gamma\widetilde{\HH})
=\bigoplus_{\gamma,\nu\in{\Lambda_{\ff}}}\mathrm{Hom}_{\widetilde{\HH}}(x_\nu\widetilde{\HH},x_\gamma\widetilde{\HH}).
\end{equation*}

Let
\[
\D_{\gamma\nu} :=\D_\gamma\cap {\D_\nu}^{-1}
\]
be the set of minimal length double coset representatives of $W_\gamma\setminus\widetilde{W}/W_\nu$.
For $\gamma,\nu\in\Lambda_{\ff}$ and $g\in\D_{\gamma\nu}$, we define an element
\begin{align*}
\phi^g_{\gamma\nu}\in\Sc_v&=\End_{\widetilde{\HH}}(\oplus_{\gamma\in{\Lambda_{\ff}}}x_\gamma\widetilde{\HH}),
\end{align*}
which is determined by
\begin{align}
 \label{eq:phig}
\phi^g_{\gamma\nu} (x_{\nu'} ) &= \delta_{\nu,\nu'} v^{\ell(w_\circ^{\nu})} H_{W_\gamma g W_\nu},\quad \forall \nu'\in\Lambda_{\ff},
\end{align}
where \begin{equation}\label{HY}
H_Y:=\sum_{w\in Y}v^{-\ell(w)}H_w\quad\mbox{for any finite subset $Y\subset \widetilde{W}$}.
\end{equation}

The following result is an affine counterpart of  a property in Dipper-James \cite{DJ86} for finite type A. As remarked by Green in \cite{Gr99}, the proof therein remains valid for infinite Coxeter groups with respect to finite parabolic subgroups; it carries over in the generality of our setting exactly as in \cite[Proposition 3.5]{LW22}.

\begin{lem} \label{lemma:basis}
The set
$\{\phi^{g}_{\gamma\nu}~|~\gamma,\nu\in\Lambda_{\ff},g\in\D_{\gamma\nu}\}$ is an $\A$-basis of $\Sc_v$.
\end{lem}

\subsection{$q$-Schur duality}

A $\widetilde{W}$-orbit is {\em regular} if the action of $\widetilde{W}$ on this orbit is transitive.
Assume that $\Lambda_{\ff}$ contains at least one regular $\widetilde{W}$-orbit, say, $\omega$. It is clear that $J_\omega=\emptyset$, $W_\omega=\{\id\}$ and $x_\omega=H_\id$. Moreover, $\D_{\omega\omega}=\D_\omega=\widetilde{W}$ and $\D_{\gamma\omega}=\D_{\gamma}$.
We calculate that
\begin{equation}
\label{eq:phiomega}
\phi_{\gamma\omega}^g(x_{\omega})=H_{W_\gamma g W_\omega}=v^{-\ell(g)-\ell(w_\circ^{\gamma})}x_\gamma H_g
\end{equation}
for any $\gamma\in\Lambda_{\ff}$ and $g\in\D_{\gamma\omega}=\D_{\gamma}$. Thus we can identify the element $v^{\ell(g)+\ell(w_\circ^{\gamma})}\phi_{\gamma\omega}^g\in\Sc_v$ with the coset $x_\gamma H_g\in x_\gamma \widetilde{\HH}\cong\T_\gamma$.

\begin{prop} [$q$-Schur duality]
Suppose that $\Lambda_{\ff}$ contains at least one regular $\widetilde{W}$-orbit, and denote the right action of $\widetilde{\HH}$ on $\T_{\ff}$ by $\Upsilon$.
Then the algebras $\Sc_v$ and $\widetilde{\HH}$ satisfy the following double centralizer property:
\begin{equation*}
\begin{array}{ll}
\Sc_v=&\End_{\widetilde{\HH}}(\T_{\ff}),\\
&\End_{\Sc_v}(\T_{\ff})=\Upsilon(\widetilde{\HH})\cong\widetilde{\HH}.
\end{array}
\end{equation*}
\end{prop}

\begin{proof}
The only nontrivial claim in the proposition is the inclusion $\End_{\Sc_v}(\T_{\ff})\subset \Upsilon(\widetilde{\HH})$.
The proof of this inclusion is similar to the proofs of \cite[Theorem 2.3.3]{Gr99} and \cite[Theorem 3.8]{LW22}, and we sketch below. Let $\psi\in\End_{\Sc_{v}}(\T_{\ff})$.
Using \eqref{eq:phiomega}, we have
$\psi(x_\gamma H_g) =v^{\ell(g)+\ell(w_\circ^{\gamma})} \psi \big(\phi^g_{\gamma \omega} (x_\omega) \big) =v^{\ell(g)+\ell(w_\circ^{\gamma})} \phi^g_{\gamma \omega} \big( \psi ( x_\omega) \big)$; that is, $\psi$ is completely determined by its value on $\psi(x_\omega)$.
 In particular, $\psi(x_\omega) = \phi^{\id}_{\omega \omega} \big( \psi ( x_\omega) \big) \in x_\omega \widetilde{\HH}.$
Therefore we have a natural identification
$\End_{\Sc_v}(\T_{\ff})\cong x_\omega\widetilde{\HH}, \psi\mapsto \psi(x_\omega).$ That is, endomorphisms in $\End_{\Sc_v}(\T_{\ff})$ are in one-to-one correspondence with elements of $\widetilde{\HH}$ acting by right multiplication.
\end{proof}

\subsection{The canonical basis for $\T_{\ff}$}

Recall $\Omega$ in \eqref{eq:Ome}. The bar involution on $\widetilde{\HH}$ is the $\mathbb{Z}$-algebra automorphism given by  (cf. \cite{Lu87b})
\[
\overline{H_k}=H_k^{-1}~ (0\leq k\leq d), \quad  \overline{g}=g ~ (g \in \Omega), \quad\mbox{and}\quad \overline{v}=v^{-1}.\]
Denote by ``$<$'' the Bruhat order on $\widetilde{W}$.

Take a $\widetilde{W}$-orbit $\gamma\in\Lambda_{\ff}$. For each $w\in \D_\gamma$, there exists  \cite{KL79,Deo87} (and see \cite{Lu87b} for passing to the extended affine Hecke algebras) a unique element $\mathcal{C}_w^\gamma\in x_\gamma\widetilde{\HH}$ such that
\begin{itemize}
\item[(1)]  $\overline{\mathcal{C}_w^\gamma}=\mathcal{C}_w^\gamma$,
\item[(2)] $\mathcal{C}_w^\gamma\in x_\gamma(H_w+\sum_{y\in \D_\gamma, y<w}v\mathbb{Z}[v]H_y)$.
\end{itemize}

The elements $\{\mathcal{C}_w^\gamma~|~w\in \D_\gamma\}$ forms an $\A$-basis of $x_\gamma\widetilde{\HH}$, which is called the canonical basis or parabolic KL basis. If $\gamma$ is regular (hence $x_\gamma=H_\id$), then we are back to the original KL basis of the affine Hecke algebra $\widetilde{\HH}$, which will be denoted by $\mathcal{C}_w$. We denote
\begin{align}
  \label{eq:KL}
 \mathcal{C}_w =\sum_{y\leq w} p_{y, w} (v) H_y,
 \end{align}
for $p_{y, w} (v)  \in v\N[v]$ (when $y< w$) and $p_{w,w} =1$.

The bar involution on $\T$ (or $\T_{\ff}$) can be characterized by
\begin{align*}
\overline{u_{\mi_\gamma}}=u_{\mi_\gamma}, &\quad\mbox{for $\gamma\in\Lambda_{\ff}$};\\
\overline{u_\mi h}=\overline{u_\mi}\overline{h},&\quad\mbox{for $\mi\in \check{X}$ (or $\check{X}_{\ff}$) and $h\in\widetilde{\HH}$}.
\end{align*}

By the identification in Lemma \ref{identification}
\begin{align*}
\Re:\T_{\ff}\cong\bigoplus_{\gamma\in\Lambda_{\ff}}x_\gamma\widetilde{\HH}, \quad u_{\mi_\gamma}\mapsto x_\gamma,
\end{align*}
the $\A$-module $\T_{\ff}$ admits a canonical basis
\begin{equation*}
\mathbf{B}(\T_{\ff})=\{\mathcal{C}_{\mi_\gamma w}:=\Re^{-1}(\mathcal{C}_w^{\gamma})~|~\gamma\in\Lambda_{\ff}, w\in \D_\gamma\}.
\end{equation*}
This basis can be characterized by the following two properties:
\begin{itemize}
\item[(1)] $\overline{\mathcal{C}_\mi}=\mathcal{C}_\mi$, for $\mi\in \check{X}_{\ff}$;
\item[(2)] $\mathcal{C}_{\mi_\gamma}\in u_{\mi_\gamma}+\sum_{y\in \D_\gamma, y<w} v\mathbb{Z}[v]u_{\mi_\gamma y}$, for $\gamma\in\Lambda_{\ff}, w\in \D_\gamma$.
\end{itemize}

\subsection{The canonical basis for $\Sc_v$}

Set
\begin{equation}
 \label{eq:Xi}
\Xi=\{ C =(\gamma,g,\nu)~|~\gamma,\nu\in\Lambda_{\ff}, g\in\D_{\gamma\nu}\}.
\end{equation}
We introduce a shorthand notation
\begin{align}  \label{eq:HC}
H_C =H_{ W_\gamma g  W_\nu},
\quad \text{ where $C =(\gamma,g,\nu)$ and $H_{ W_\gamma g  W_\nu}$ is defined in \eqref{HY}}.
\end{align}
For $C=(\gamma,g,\nu)\in\Xi$, let $g_{\gamma\nu}^+$ be the longest element in $ W_\gamma g W_\nu$. In particular, $\id_{\nu\nu}^+=w_\circ^\nu$ is the longest element in $ W_\nu$. (A variant of) the following can be found in \cite{Cur85}.

\begin{lem}\label{eq:xandC}
Let $C =(\gamma,g,\nu)\in\Xi$. Then we have
\begin{itemize}
\item[(1)] $ W_\gamma g  W_\nu = \{w \in W ~|~ g \leq w \leq g^+_{\gamma\nu}\}$;

\item[(2)] $H_C
= v^{-\ell(g^+_{\gamma\nu})} \mathcal{C}_{g^+_{\gamma\nu}} + \sum_{\substack{y\in \D_{\gamma\nu}\\
 y < g }} c^{(\gamma,\nu)}_{y,g} \mathcal{C}_{y^+_{\gamma\nu}}$,
for $c^{(\gamma,\nu)}_{y,g}\in\A$.
\end{itemize}
\end{lem}

We define a bar involution $\bar{\phantom{x}}$ on $\Sc_{v}$ as follows: for each $f \in \Hom_{\widetilde{\HH}}(x_\nu \widetilde{\HH}, x_\gamma \widetilde{\HH}) \subset \Sc_{v}$, let $\overline{f}\in\Hom_{\widetilde{\HH}}(x_\nu \widetilde{\HH}, x_\gamma \widetilde{\HH}) \subset \Sc_{v}$ be the
$\widetilde{\HH}$-linear map which sends $x_\nu=\mathcal{C}_{w_\circ^\nu}$ to $\overline{f(\mathcal{C}_{w_\circ^\nu})}$. That is, we have
\begin{equation}
    \label{eq:barS}
\overline{f}(x_{\nu'} h) = \delta_{\nu',\nu} \overline{f(x_\nu)} h,
\quad
\mbox{for $h \in \widetilde{\HH}$}.
\end{equation}
Hence it follows from Lemma \ref{eq:xandC} that
\begin{eqnarray}
\phi_{\gamma\nu}^g(\mathcal{C}_{w_\circ^\nu})
&=& v^{\ell(w_\circ^\nu)-\ell(g^+_{\gamma\nu})} \mathcal{C}_{g^+_{\gamma\nu}}
 + \sum_{\substack{y\in\D_{\gamma\nu}\\y < g}}
v^{\ell(w_\circ^\nu)} c_{y,g}^{(\gamma,\nu)} \mathcal{C}_{y^+_{\gamma\nu}},
    \label{eq:eA}
\\
\overline{\phi_{\gamma\nu}^g}(\mathcal{C}_{w_\circ^\nu})
&=& v^{\ell(g^+_{\gamma\nu})-\ell(w_\circ^\nu)} \mathcal{C}_{g^+_{\gamma\nu}}
+ \sum_{\substack{y\in\D_{\gamma\nu}\\y < g}}
v^{-\ell(w_\circ^\nu)}\overline{c_{y,g}^{(\gamma,\nu)}} \mathcal{C}_{y^+_{\gamma\nu}}.
    \label{eq:eAbar}
\end{eqnarray}

For any $(\gamma,g,\nu)\in\Xi$, we set
\begin{equation}\label{def:dA}
[\phi_{\gamma\nu}^g]=v^{\ell(g_{\gamma\nu}^+)-\ell(w_\circ^\nu)}\phi_{\gamma\nu}^g.
\end{equation}
Then $\{[\phi_{\gamma\nu}^g] \mid (\gamma,g,\nu)\in\Xi \}$ forms an $\A$-basis for $\Sc_v$, which is called a {\em standard basis}. Thanks to \eqref{eq:eA} and \eqref{eq:eAbar}, we have
\begin{equation*}
\overline{[\phi_{\gamma\nu}^g]}\in [\phi_{\gamma\nu}^g]+\sum_{g>y\in\D_{\gamma\nu}}\A[\phi_{\gamma\nu}^y].
\end{equation*}

Similar to \cite{Du92}, we define
\begin{equation*}
   \{\phi_{\gamma\nu}^g\} \in  \Hom_{\widetilde{\HH}}(x_\nu\widetilde{\HH}, x_\gamma\widetilde{\HH}), \text{  and hence } \{\phi_{\gamma\nu}^g\} \in \Sc_{v},
\end{equation*}
   by requiring
\begin{equation*}
  \{\phi_{\gamma\nu}^g\} (\mathcal{C}_{w^{\nu}_\circ}) = \mathcal{C}_{g^+_{\gamma\nu}}.
\end{equation*}
It follows by \eqref{eq:barS} that $\{\phi_{\gamma\nu}^g\}$ is bar invariant, i.e.,
\begin{equation}\label{barA}
     \overline{\{\phi_{\gamma\nu}^g\}} = \{\phi_{\gamma\nu}^g\}.
\end{equation}

Following \cite[(2.c), Lemma~3.8]{Du92}, we have
\begin{equation}   \label{eq:canonical}
\{\phi_{\gamma\nu}^g\} \in [\phi_{\gamma\nu}^g] + \sum_{y< g}v\Z[v]   \, [\phi_{\gamma\nu}^y].
\end{equation}
More precisely, we have (cf. \cite{Cur85, Du92})
\begin{equation}   \label{eq:CB1}
\{\phi_{\gamma\nu}^g\} =  [\phi_{\gamma\nu}^g] + \sum_{y< g} p_{y^+_{\gamma\nu},g^+_{\gamma\nu}} [\phi_{\gamma\nu}^y].
\end{equation}
where $p_{y^+_{\gamma\nu},g^+_{\gamma\nu}} \in v\Z[v]$ are Kazhdan-Lusztig polynomials. 

By Lemma~ \ref{lemma:basis}, \eqref{def:dA} and \eqref{eq:canonical}, the set
$$\B (\Sc_v) =\big\{\{\phi_{\gamma\nu}^g\}\ |\ (\gamma,g,\nu) \in \Xi \big\}$$ forms an $\A$-basis of $\Sc_v$, which is called {\it the canonical basis}.
We summarize this as follows.

\begin{prop}
There exists a canonical basis $\B (\Sc_v)=\big\{\{\phi_{\gamma\nu}^g\}\ |\ (\gamma,g,\nu) \in \Xi \big\}$ for $\Sc_v$,
which is characterized by the properties \eqref{barA}--\eqref{eq:canonical}.
\end{prop}

\section{A geometric setting for affine $q$-Schur algebras}
\label{sec:geom}
\subsection{Affine flag varieties}

Let $\mathbb{F}_{q}$ be a finite field of ${q}$ elements of characteristic $>3$, where ${q}$ is a prime power. Let ${\mathbb K}=\mathbb{F}_{{q}}((\varepsilon))$ be the field of formal Laurent series over $\mathbb{F}_{{q}}$ and $\mathfrak{o}=\mathbb{F}_{{q}}[[\varepsilon]]$ the ring of formal power series. Let $\mathbf{G}({\mathbb K})$ (resp. $\mathbf{G}(\mathbb{F}_{q})$, $\mathbf{G}(\mathfrak{o})$) be a connected algebraic group of rank $d$ defined over ${\mathbb K}$ (resp. $\mathbb{F}_{{q}}$, $\mathfrak{o}$). Write $G=\mathbf{G}({\mathbb K})$. Fix a Borel subgroup $B\subset\mathbf{G}(\mathbb{F}_{{q}})$, and we have an Iwahori subgroup $\Iw\subset G$ corresponding to $B$. Associated to each proper subset $J\subset\{0,1,\ldots,d\}$, we have a standard parahoric subgroup $P_J$. In particular, $P_{\emptyset}=\Iw$ and $P_{\{1,\ldots,d\}}=\mathbf{G}(\mathfrak{o})$. Recall $\Lambda_{\ff}$ from \eqref{lambdaf} and $J_\gamma$ from \eqref{Jgamma}. For each $\gamma\in \Lambda_{\ff}$, we denote $P_\gamma=P_{J_\gamma}$.

Set
$$\mathscr{F}=\bigsqcup_{\gamma\in\Lambda_{\ff}}G/P_\gamma,\quad\qquad \mathscr{B}=G/\Iw.$$
(Here $\bigsqcup$ means disjoint union.)

There is a natural action of $G$ on $\mathscr{F}$ and on $\mathscr{B}$. So $G$ can act diagonally on $\mathscr{F}\times\mathscr{B}$, $\mathscr{F}\times\mathscr{B}$ and $\mathscr{F}\times\mathscr{B}$, respectively. Recall $\Xi$ from \eqref{eq:Xi}. The following lemma is standard, see \cite{IM65, BLM90, GL92}  and also see \cite[Lemma 4.1]{LW22}.

\begin{lem}
\label{lem:bijections}
There are natural bijections:
\begin{equation*}
G\setminus(\mathscr{F}\times\mathscr{F})\longleftrightarrow\Xi,\quad G\setminus(\mathscr{F}\times\mathscr{B})\longleftrightarrow \check{X}_{\ff},\quad
G\setminus(\mathscr{B}\times\mathscr{B})\longleftrightarrow \widetilde{W}
\end{equation*}
\end{lem}

Thanks to the above lemma, We shall denote these $G$-orbits by $\mathscr{O}_C, \mathscr{O}_{\mi}$ and $\mathscr{O}_w$, respectively.

\subsection{Geometric $q$-Schur duality}
\label{subsec:GeoSchur}

Define
\begin{equation*}
\mathbf{S}'_{q}:=\Z_G(\mathscr{F}\times\mathscr{F}), \quad
\mathbf{T}'_{\ff,{q}}:=\Z_G(\mathscr{F}\times\mathscr{B}), \quad
\widetilde{\HH}_{q}':=\Z_G(\mathscr{B}\times\mathscr{B})
\end{equation*}
to be the space of $G$-invariant $\Z$-valued functions on $\mathscr{F}\times\mathscr{F}$, $\mathscr{F}\times\mathscr{B}$ and $\mathscr{B}\times\mathscr{B}$, respectively. We denote by $\phi_{C,q}$ (resp. $u'_{\mi,q}$ and $H'_{w,q}$) the characteristic function of the orbit $\mathscr{O}_C$ (resp. $\mathscr{O}_{\mi}$ and $\mathscr{O}_w$).

It is well known that there is a convolution product on $\widetilde{\HH}_{q}'$, which gives us a $\Z$-algebra structure on $\widetilde{\HH}_{q}'$ identified with the Hecke algebra $\widetilde{\mathbf{H}}$ specialized at $v =\sqrt{{q}}^{-1}$ (cf. \cite{IM65}).
A similar convolution product gives us a $\Z$-algebra structure on $\mathbf{S}'_{q}$.
 Moreover, there are a left $\mathbf{S}'_{q}$-action on $\mathbf{T}'_{\ff,{q}}$ and a right $\widetilde{\HH}_{q}'$-action on $\mathbf{T}'_{\ff,{q}}$ induced by convolution products. By standard arguments which goes back to \cite{IM65} (also see \cite{Lu99}), the structure constants for the algebras $\widetilde{\HH}_{q}'$, and $\mathbf{S}'_{q}$ as well as for their bimodule action on $\mathbf{T}'_{\ff,{q}}$
with respect to the basis $\{H'_{w,q}\}$, $\{\phi_{C,q}\}$ and $\{u'_{\mi,q}\}$
are polynomials in ${q}$. Therefore, a substitution ${q} \leadsto v^{-2}$ allows us to define generic versions of $\widetilde{\HH}_{q}'$, $\mathbf{S}'_{q}$ and $\mathbf{T}'_{\ff,{q}}$ over $\A =\Z[v, v^{-1}]$, denoted by $\widetilde{\HH}', \mathbf{S}'_v,
\mathbf{T}'_{\ff}$, respectively.
By Lemma~\ref{lem:bijections}, $\mathbf{S}'_v$ (resp. $\mathbf{T}'_{\ff}$ and $\widetilde{\mathbf{H}}'$) is a free $\A$-module with a basis $\{\phi_C~|~C\in\Xi\}$ (resp. $\{u'_\mi~|~\mi\in \check{X}_{\ff}\}$ and $\{H'_w~|~w\in \widetilde{W}\}$).

Here we precisely describe the convolution product on $\mathbf{S}'_{q}$ for latter use. For triple $(A,B,C)\in\Xi\times\Xi\times\Xi$, choose $(f_1,f_2)\in \mathscr{O}_C$ and let $m_{A,B;q}^C$ be the number of $f\in\sF$ such that $(f_1,f)\in\mathscr{O}_A$ and $(f,f_2)\in\mathscr{O}_B$, then
\begin{equation}
\label{eq:madd}
\phi_{A,q}\phi_{B,q}=\sum_{C\in\Xi}m_{A,B;q}^C\phi_{C,q},
\end{equation}
and
\begin{equation}
\label{eq:m}
 \phi_{A}\phi_{B}=\sum_{C\in\Xi}m_{A,B}^C\phi_{C}
\end{equation}
where $m_{A,B}^C=m_{A,B;q}^C|_{q\leadsto v^{-2}}\in\A$.

\begin{thm}\label{iso}
We have the following commutative diagram:
\begin{eqnarray}
  \label{CD}
\begin{array}{ccccc}
  \mathbf{S}'_v &\circlearrowright
 & \mathbf{T}_{\ff}' & \circlearrowleft  \; & \widetilde{\mathbf{H}}'
  \\
  \aleph \downarrow \simeq\;\; & & \downarrow \simeq& & \parallel
 \\
\mathbf{S}_v     & \circlearrowright
& \mathbf{T}_{\ff} &  \circlearrowleft  \; & \widetilde{\mathbf{H}}
\end{array}
 \end{eqnarray}
 where the identifications are given by
 \begin{align*}
 \phi_C\mapsto \phi_{\gamma\nu}^g, \quad & \mbox{for $C=(\gamma,g,\nu)\in\Xi$},\\
 u'_\mi \mapsto v^{-\ell(\sigma)}u_\mi, \quad & \mbox{for $\mi=\mi_\nu\sigma$ with $\sigma\in \D_{\nu}$},\\
 H'_w \mapsto v^{-\ell(w)}H_w, \quad & \mbox{for $w\in W$}.
 \end{align*}
This provides a geometric realization of the q-Schur duality if $\Lambda_{\ff}$ contains at least one regular $\widetilde{W}$-orbit.
\end{thm}

\begin{proof}
We shall use an index $q$ to denote the specialization at $v\mapsto q^{-\frac12}$ or the elements defined over the finite field $\mathbb F_q$ in this proof.

The isomorphism $\widetilde{\HH}'_q \cong \widetilde{\HH}_q$ given in \cite{IM65} is built on some formal aspects of BN-pairs or Tits systems; cf. \cite[Chapter V]{Ku02} or \cite[Theorem 4.37]{DDPW}, where $\mathfrak S_n$ is replaced by $\widetilde W$ here. The transitive $G$-set $G/\Iw$ gives rise to a permutation $G$-module over $\Z$, denoted by $\mathrm{Ind}_{\Iw}^{G} \Z$. The module  $\mathrm{Ind}_{\Iw}^{G} \Z$ admits a natural basis $\{ \underline{g \Iw} \}$ indexed by the distinct left cosets $g \Iw$ of $\Iw$ in $G$ and is generated by $\underline{\Iw}$. Note, for any $w \in \widetilde W$, $\Iw w \Iw =U_w \Iw =\sqcup_{u\in U_w} u\Iw$ is a finite union of left cosets of $\Iw$, where $U_w$  is a unipotent subgroup of $G$ with cardinality $q^{\ell(w)}$. We extend the notation $\underline{X} =\sum_{a \in X_0} \underline{a\Iw}  \in \mathrm{Ind}_{\Iw}^{G} \Z$  for a finite union of left cosets $X= \sqcup_{a \in X_0} a \Iw$ of $\Iw$ in $G$.
In particular, we have $\underline{\Iw w\Iw} := \sum_{u\in U_w} \underline{u \Iw}  \in \mathrm{Ind}_{\Iw}^{G} \Z$. We have the following identifications \cite{IM65}:
\[
\widetilde{\HH}'_q=\Z_G(\mathscr{B}\times\mathscr{B})
\simeq\End_G\big(\mathrm{Ind}_{\Iw}^{G}\mathbb{Z}\big)^\mathrm{op}.
\]
In the above notations, $H'_{w,q}$ acts on the generator $\underline{\Iw}$ in $\mathrm{Ind}_{\Iw}^{G} \Z$ by $H'_{w,q} \underline{\Iw} = \underline{\Iw w\Iw}$.

Following the same argument as in the proof of \cite[Theorem~13.15]{DDPW}, one sees that $\aleph: \Sc'_q\rightarrow \Sc_q, \phi_C\mapsto \phi_{\gamma\nu}^g$, for $C=(\gamma,g,\nu)$, is a $\Z$-module isomorphism by verifying that $\phi_C(\underline{P_\gamma}) =\phi_{\gamma,\nu}^g (x'_\nu)(\underline{\Iw})$. Below we shall verify that the map $\aleph$  is an algebra isomorphism, adapting the proof of \cite[Theorem~13.15]{DDPW} (which is valid in finite type). Here and below in this proof we denote $x'_\mu =q^{\frac12\ell(w_\circ^\mu)} x_\mu$.


By the Bruhat decomposition we have
$
P_\gamma=\sqcup_{w\in W_\gamma}\Iw w\Iw$ and
$P_\gamma g P_\nu=\sqcup_{w\in W_\gamma g W_\nu}\Iw w\Iw$. Hence, we have
\[
\underline{P_\gamma} =\sum_{w\in W_\gamma}H'_{w,q} \underline{\Iw},
\qquad
\underline{P_\gamma g P_\nu} =\sum_{w\in W_\gamma g W_\nu} H'_{w,q} \underline{\Iw}.
\]

Note that
\[
\Sc'_q =\Z_G(\mathscr{F}\times\mathscr{F})
\simeq \End_G \Big(\bigoplus_{\lambda\in{\Lambda_\ff}} \mathrm{Ind}_{P_\lambda}^{G}\mathbb{Z}\Big)^\mathrm{op}.
\]
The $G$-orbit $\mathscr{O}_C$ in $\mathscr{F}\times \mathscr{F}$ contains the element $(P_\gamma/P_{\gamma},g P_\nu/P_\nu)$, where $C=(\gamma,g,\nu)\in\Xi$, and hence $\phi_{C,q}$ acts on the generator $\underline{P_\lambda}$ of the $G$-module $\mathrm{Ind}_{P_\lambda}^{G}\mathbb{Z}$ by
\begin{equation*}
\phi_{C,q}(\underline{P_\lambda})=\delta_{\gamma\lambda}\underline{P_\gamma g P_\nu}.
\end{equation*}
A similar argument to \cite[pp.549]{DDPW} shows that
\[
\underline{P_\gamma g P_\nu}
=\wp \cdot \underline{P_\nu}, \qquad\mbox{with}\quad \wp
=\sum_{y\in(W_\gamma)^\chi} \Big(\sum_{u \in U_{yg}} u \Big) yg,
\]
where $\chi\in\Lambda_\ff$ is such that $W_\chi=W_{\gamma}\cap gW_\nu g^{-1}$, and $(W_\gamma)^\chi=\{y\in W_\gamma~|~\ell(wy)=\ell(w)+\ell(y), \forall w\in W_\chi\}$.

For $A=(\nu,w,\mu)$ and $B=(\gamma,g,\nu)$ in $\Xi$, we have
\begin{align*}
\phi_{A,q}\phi_{B,q}(\underline{P_\gamma})&=\phi_{A,q}(\underline{P_\gamma g P_\nu})=\wp\cdot
\phi_{A,q}(\underline{P_\nu})\\
&=\wp\cdot \underline{P_\nu w P_\mu}=\wp\cdot \sum_{z\in W_\nu w W_\mu}\underline{\Iw z\Iw}.
\end{align*}
Thus we have
\begin{align*}
\phi_{A,q}\phi_{B,q}(\underline{P_\gamma})&=\wp\cdot \Iw \sum_{z\in W_\nu w W_\mu}\underline{\Iw z\Iw}=\sum_{y\in (W_\gamma)^\chi} \underline{\Iw yw\Iw}\sum_{z\in W_\nu w W_\mu}\underline{\Iw z\Iw}\\
&=\sum_{y\in (W_\gamma)^\chi}H'_{yw,q}\sum_{z\in W_\nu w W_\mu}H'_{z,q}(\underline{\Iw})=\phi_{\gamma\nu}^g\phi_{\nu\mu}^w(x'_\mu)(\underline{\Iw}).
\end{align*}
Therefore $\aleph: \Sc'_q\simeq\Sc_q$ is an algebra isomorphism.

Similarly, we can prove the identification for $\T'_\ff\simeq\T_\ff$ and its compatibility with the convolution products which lead to the left Schur algebra action and the right Hecke algebra action.
\end{proof}

We still denote by $\phi_C\in\Sc'_q$ and $\phi_{\gamma\nu}^g\in\Sc_q$ the elements corresponding to $\phi_C\in\Sc'_v$ and $\phi_{\gamma\nu}^g\in\Sc_v$, respectively.

\subsection{Canonical basis and positivity}

For $C=(\gamma,g,\nu) \in\Xi$ and $L\in\mathscr{F}$, we define
\begin{align}
  \label{eq:XC}
X_C^L:=\{L'\in\mathscr{F}~|~(L,L')\in\mathscr{O}_{C}\}.
\end{align}
Recall $g_{\gamma\nu}^+$ is the longest element in $ W_\gamma g W_\nu$. The dimension of the Zariski closure $\overline{X}_C^L(\C)$ (a generalized Schurbert variety in $G/P_\nu$), which is independent of the choice of $L$, is equal to the length of the minimal length representative in the coset $g^+_{\gamma\nu} W_\nu$, which can be written as follows.

\begin{lem}  \label{eq:dell}
For any $C=(\gamma,g,\nu)\in\Xi$, the dimension of the variety $\overline{X}_C^L(\C)$ is  $\ell(g^+_{\gamma\nu})-\ell(w_\circ^\nu).$
\end{lem}

Lemma~\ref{eq:dell} allows us to define the following map
\begin{align}
  \label{eq:dC}
\begin{split}
d_{\cdot} : \Xi  & \longrightarrow  \mathbb{N},\\
 C=(\gamma,g,\nu) & \mapsto  d_C := \ell(g^+_{\gamma\nu})-\ell(w_\circ^\nu).
 \end{split}
\end{align}

We shall call
\begin{align}
 \label{eq:Sbasis}
 \{[C]:=v^{d_C}\phi_C~|~C\in\Xi\}
 \end{align}
 a {\em standard basis}  for $\mathbf{S}'_v$.
Lemma~\ref{eq:dell} and \eqref{def:dA} imply that the algebra isomorphism $\aleph: \mathbf{S}'_v\longrightarrow\mathbf{S}_v$ in \eqref{CD}
 satisfies
 \begin{align}
  \label{eq:Theta2}
  \aleph([C])=[\phi_{\gamma\nu}^g].
 \end{align}

Let $\text{IC}_C$, for $C\in \Xi$, be  the shifted intersection complex associated with the closure of  the orbit $\Ob_C$
such that the restriction of $\text{IC}_C$ to $\Ob_C$ is the constant sheaf on  $\Ob_C$.
Since $\text{IC}_C$ is $G$-equivariant, the stalks of the $i$-th cohomology sheaf of $\text{IC}_C$  at different points in $\Ob_{C'}$ (for $C' \in \Xi$) are isomorphic. Let $\mathscr H^i_{\Ob_{C'}} (\text{IC}_C)$ denote the stalk of the $i$-th cohomology sheaf of $\text{IC}_C$ at any point in $\Ob_{C'}$.
We  set
\begin{align}
 \label{eq:A}
\begin{split}
P_{C', C} &=\sum_{i\in \Z} \dim \mathscr H^i_{\Ob_{C'}} (\text{IC}_C) \; v^{-i + d_C -d_{C'}},
 \\
\{ C\} &= \sum_{C'\leq C} P_{C', C} [C'],
\end{split}
\end{align}
where the partial order ``$<$'' on $\Xi$ is defined as follows:
$$(\gamma,g,\nu)<(\gamma',g',\nu') \quad\mbox{if and only if } \gamma=\gamma', \nu=\nu', g<g'.$$

By the properties of intersection complexes, we have
\begin{equation*}
P_{C, C} =1, \qquad P_{C', C} \in v \N [v] \; \; \mbox{ for } C' < C.
\end{equation*}
As in \cite[1.4]{BLM90}, we have an anti-linear bar involution $\bar\ : \Scg \to \Scg$ such that
\[
\overline{\{C\}} =\{C\}, \quad \forall C\in \Xi.
\]
We refer to $\mathbf{B}(\mathbf{S}'_v)=\{ \{C\} \mid C\in \Xi\}$ as the {\em canonical basis} for $\mathbf{S}'_v$. Recall from \eqref{eq:CB1} the Kazhdan-Lusztig polynomials $p_{y_{\gamma\nu}^{+},g_{\gamma\nu}^{+}}$ and the canonical basis element $\{\phi_{\gamma\nu}^g\}$ in $\mathbf{S}_v$.

\begin{prop}
\label{addprop:alg-isomor}
The algebra isomorphism
$\aleph: \mathbf{S}'_v\rightarrow\mathbf{S}_v$ in \eqref{CD} matches the canonical bases $\mathbf{B}(\mathbf{S}'_v)$ and $\mathbf{B}(\mathbf{S}_v)$, i.e., $\aleph(\{C\})=\{\phi_{\gamma\nu}^g\}$. In particular, we have
\begin{equation}
\label{eq:p=P}
P_{C',C}=p_{y_{\gamma\nu}^{+},g_{\gamma\nu}^{+}} \quad\mbox{for } C'=(\gamma,y,\nu)\leq C=(\gamma,g,\nu).
\end{equation}

\begin{proof}
Recall from \eqref{eq:Theta2} that $\aleph([C])=[\phi_{\gamma\nu}^g].$
Now the proposition follows by comparing \eqref{eq:CB1} and \eqref{eq:A} and the uniqueness of canonical basis.
\end{proof}
\end{prop}

For any $A, B\in \Xi$ and any canonical basis element $u\in\mathbf{B}(\mathbf{T}_{\ff})$, we denote
\begin{equation}
\label{eq:positivity}
\{A\}\{B\}=\sum_{C\in \Xi} g_{A,B}^{C} \{C\}, \qquad
\aleph(\{A\})\cdot u=\sum_{u'\in\mathbf{B}(\mathbf{T}_{\ff})}t_{A,u}^{u'}{u'}
\end{equation}
for $g_{A,B}^{C},\ t_{A,u}^{u'}\in \A$. The following positivity property (which generalizes \cite{Lu99, LiW18}) follows from the geometric interpretation of these canonical bases and their multiplication/action in terms of perverse sheaves and their convolution products.

\begin{prop}  [Positivity]
 \label{prop:positive}
Retain the notations \eqref{eq:positivity}. Then  $g_{A,B}^{C},\ t_{A,u}^{u'}\in \mathbb{N}[v,v^{-1}]$.
\end{prop}

\subsection{An anti-automorphism of $\mathbf{S}'_v$}
For any $C=(\gamma,g,\nu)\in \Xi$, we shall denote
\begin{equation}
\label{eq:Ct}
C^t=(\nu,g^{-1},\gamma)
\end{equation} which still lies in $\Xi$, and denote
\begin{equation*}
\ro(C) =\gamma, \qquad \co(C)=\nu;
\end{equation*}
indeed $\gamma$ and $\nu$ are row/column vectors of some matrices in type A \cite{BLM90, Lu99} or type B/C \cite{BKLW18, FLLLWa}. Furthermore, we shall sometimes denote
\begin{equation}
\label{eq:wC}
w_C=g \quad \mbox{and}\quad w_C^+=g_{\gamma\nu}^+.
\end{equation}
It is obvious that
\begin{equation}
  \label{eq:lw}
\ell(w_C)=\ell(w_{C^t}),\quad  \ell(w_C^+)=\ell(w_{C^t}^+).
\end{equation}
Furthermore, we have
\begin{equation}
\label{eq:PCC}
P_{C',C}=P_{C'^t, C^t} \quad\mbox{for any } C'\leq C,
\end{equation}
which follows from \eqref{eq:p=P} together with the fact that $p_{y^{-1},w^{-1}}=p_{y,w}$ $(\forall y, w\in \widetilde{W})$ (cf. \cite[\S5.6]{Lu03}).

\begin{lem}
The $\A$-module homomorphism $\Psi: \mathbf{S}'_v\rightarrow \mathbf{S}'_v$ defined by
\begin{equation}
\label{eq:psi}
\Psi([C])=[C^t] \quad\quad(\forall C\in\Xi)
\end{equation}
is an anti-automorphism of $\mathbf{S}'_v$. Moreover, $\Psi(\{C\})=\{C^t\}$.
\end{lem}
\begin{proof}
Take any $A,B\in\Xi$ with $\co(A)=\ro(B)$. We have
\begin{align*}
\Psi([A][B])&=\Psi(\sum_{C}v^{d_A+d_B-d_C}m_{A,B}^C[C]) \quad\quad\quad \mbox{by \eqref{eq:m} and \eqref{eq:Sbasis}}\\
&=\sum_{\ro(C)=\ro(A), \co(C)=\co(B)}v^{d_A+d_B-d_C}m_{A,B}^C[C^t]
\end{align*}
and
\begin{align*}
\Psi([B])\Psi([A])&=[B^t][A^t]=\sum_C v^{d_{A^t}+d_{B^t}-d_C}m_{B^t,A^t}^C[C] \quad\quad \mbox{by \eqref{eq:m} and \eqref{eq:Sbasis}}\\
&=\sum_{\ro(C)=\ro(A), \co(C)=\co(B)} v^{d_{A^t}+d_{B^t}-d_{C^t}}m_{B^t,A^t}^{C^t}[C^t]
\end{align*}

From the construction of the convolution product on $\mathbf{S}'_v$, we see that the structure constants $m_{A,B}^C$ $(A,B,C\in\Xi)$ satisfy
$$m_{A,B}^C=m_{B^t,A^t}^{C^t}.$$
Furthermore, if $\co(A)=\ro(B), \ro(C)=\ro(A)$ and $\co(C)=\co(B)$, we have
\begin{align*}
d_A+d_B-d_C&=\ell(w_A^+)-\ell(w_\circ^{\co(A)})+\ell(w_B^+)-\ell(w_\circ^{\co(B)})-\ell(w_C^+)+\ell(w_\circ^{\co(C)})\\
& \qquad\qquad\qquad\qquad\qquad\qquad\qquad \mbox{by \eqref{eq:dC} and \eqref{eq:wC}} \\
&=\ell(w_{A^t}^+)-\ell(w_\circ^{\co(B^t)})+\ell(w_{B^t}^+)-\ell(w_{C^t}^+)
\qquad\qquad \mbox{by \eqref{eq:lw}}\\
&=\ell(w_{A^t}^+)-\ell(w_\circ^{\co(A^t)})+\ell(w_{B^t}^+)-\ell(w_\circ^{\co(B^t)})-\ell(w_{C^t}^+)+\ell(w_\circ^{\co(C^t)})\\
&=d_{A^t}+d_{B^t}-d_{C^t}.
\end{align*}
Therefore, we have $\Psi([A][B])=\Psi([B])\Psi([A])$, which implies $\Psi$ is an anti-automorphism of $\mathbf{S}'_v$.

Finally,
\begin{align*}
  \Psi(\{C\})&=\Psi(\sum_{D\leq C}P_{D,C}[D])=\sum_{D\leq C}P_{D,C}[D^t]\\
  &=\sum_{D\leq C}P_{D^t,C^t}[D^t] \qquad\qquad\mbox{by \eqref{eq:PCC}}
  \\&=\sum_{D\leq C^t}P_{D,C^t}[D]=\{C^t\}.
\end{align*}
\end{proof}

\section{Positivity of an inner product on affine $q$-Schur algebras}
\label{sec:inner}
Following Lusztig \cite{Lu99, Lu00} (cf. \cite[(3-3)]{Mc12}) in affine type A, we define a natural inner product $(\cdot,\cdot)$ on $\Scg$.

\subsection{Adjunction property}

Recall $\phi_C$ is the characteristic function of the orbit $\Ob_C$ and recall $X_C^L$ from \eqref{eq:XC}. We define an inner product  $(\cdot, \cdot)_{{q}} :\mathbf{S}_{{q}}'\times \mathbf{S}_{{q}}'\rightarrow \Q$ by
\[
(\phi_C, \phi_{C'})_{{q}}=\delta_{C,C'}{q}^{d_{C}-d_{C^{t}}}|X_{C^{t}}^{L'}|.
\]
where $|X_{C^{t}}^{L'}|$ is the number of points of the variety $X_{C^{t}}^{L'}$ over $F_q$.
We still denote by $[C]\in\Sc'_q$ the element corresponding to $[C]\in\Sc'_v$. Let $f_{A, B; {q}}^{C}$ be the structure constants of $\mathbf{S}_{{q}}'$ with respect to the basis $\{ [C] \mid C\in \Xi\}$.
It follows from \eqref{eq:madd} and the definitions of $X_{C^{t}}^{L'}$ and $f_{A, B; q}^{C}$ that  (cf. \cite{Mc12})
$$|X_{C^{t}}^{L'}|=q^{\frac{d_C+d_{C^t}}{2}}f_{C^{t}, C; q}^{(\nu,\id,\nu)},\quad \mbox{where}\ \nu=\co(C).$$
Therefore we have
\begin{align}\label{bilinear-form-1}
([C], [{C'}])_{{q}}=\delta_{C,C'}{q}^{\frac{d_{C}-d_{C^{t}}}{2}}f_{C^{t}, C; {q}}^{(\nu,\id,\nu)}, \quad\mbox{where}\ \nu=\co(C).
\end{align}

Let $f_{A, B}^{C}$ denote the structure constants of $\Scg$ with respect to the standard basis $\{ [C] \mid C\in \Xi\}$, i.e.,
\begin{align}
\label{eq:fABC}
[A] \cdot [B] = \sum_{C\in \Xi} f_{A, B}^{C} [C].
\end{align}

The inner product $(\cdot, \cdot)_{{q}}$ and \eqref{bilinear-form-1} induce an inner product on $\Scg$:
\begin{align}
   \label{bilinear-form-2}
   \begin{split}
 (\cdot, \cdot): & \; \Scg \times \Scg \longrightarrow \A
 \\
([C], [C']) & =\delta_{C,C'}v^{d_{C^{t}} -d_{C}}f_{C^{t}, C}^{(\nu,\id,\nu)},\quad\mbox{where } \nu=\co(C).
\end{split}
\end{align}

We define an $\A$-module homomorphism $\rho$ on $\Scg$ such that
\begin{align}
  \label{eq:rho1}
 \rho([C])=v^{d_{C^{t}} -d_{C}}[C^{t}].
\end{align}

\begin{lem}
\label{addlem:anti-automorphismrho}
For any $C \in \Xi$, we have
\begin{align}
  \label{eq:rho2}
\rho(\{C\})=v^{d_{C^{t}} -d_{C}}\{C^{t}\}.
\end{align}
\end{lem}

\begin{proof}
By \eqref{eq:PCC} we have $P_{C',C}=P_{C'^t, C^t}$ for any $C'\leq C$. Furthermore, since $d_{C^{t}} -d_{C}= \ell(w_\circ^\nu)-\ell(w_\circ^\gamma)$ for $C=(\gamma,g,\nu)\in \Xi$, which does not depend on $g$, we have $d_{C^{t}} -d_{C}=d_{C'^{t}} -d_{C'}$ for any $C'\leq C$. Hence
\begin{align*}
\rho(\{C\})&=\rho(\sum_{C'\leq C}P_{C',C}[C'])=\sum_{C'\leq C}P_{C',C}v^{d_{C'^t}-d_{C'}}[C'^t]\\
&=v^{d_{C^t}-d_{C}}\sum_{C'^t\leq C^t}P_{C'^t,C^t}[C'^t]=v^{d_{C^t}-d_{C}}\{C^t\}.
\end{align*}
\end{proof}

\begin{lem}
\label{lem:d=d00-aa}
For any $A, B, C\in \Xi$, we have $(\{A\}\{B\}, \{C\})=(\{B\}, \rho(\{A\})\{C\})$.
\end{lem}

\begin{proof}
It suffices to prove the same identity using the inner product in $\mathbf{S}_{{q}}'$. Since the characteristic functions form a basis of $\mathbf{S}_{{q}}'$, it suffices to prove that $$([A]\phi_B, \phi_C)_{{q}}={q}^{\frac{d_{A}-d_{A^{t}}}{2}}(\phi_B, [A^{t}]\phi_C)_{{q}}.$$ Moreover we may assume that $A=(\gamma, w_{A}, \mu),$ $B=(\mu, w_{B}, \nu)$ and $C=(\gamma, w_{C}, \nu),$ as otherwise both sides of the above identity are zero. The remaining part of the proof is standard (which can be found in \cite[Proposition 3.2]{Mc12}), and will be skipped.
\end{proof}

\subsection{Positivity of the inner product}

Note that $x_J$ in \eqref{eq:xj} is a canonical basis element, i.e., $x_J =\mathcal{C}_{w_J}$. For proper subsets $I,J \subset \{0,1,\ldots, d\}$, we define ${}^I\! \widetilde{\HH}^J :=x_I \widetilde{\HH} \, x_J$.  The $\A$-linear anti-involution $i: \widetilde{\HH} \rightarrow \widetilde{\HH}, H_x \mapsto H_{x^{-1}}$, induces an isomorphism of $\A$-modules
\[
i: {}^I\! \widetilde{\HH}^J \longrightarrow {}^J\! \widetilde{\HH}^I.
\]
We define the following bar invariant element in $\A$:
\begin{align}
  \label{eq:piJ}
\pi(J) =v^{\ell(w_\circ^J)} \sum_{w\in W_{J}}v^{-2\ell(w)}.
\end{align}
Following and {\em slightly normalizing} \cite[pp.4573, 4568]{Wi11}, we define an inner product
\begin{align}
 \label{eq:bform}
\langle \cdot, \cdot \rangle :    {}^I\! \widetilde{\HH}^J \times {}^I\! \widetilde{\HH}^J \longrightarrow \A,
 \quad
\langle h, g \rangle  = \text{coefficient of } H_\id \text{ in }
\frac1{v^{\ell(w_\circ^J)}\pi(J)} h \cdot  i(g).
\end{align}
We remark that the inner product defined in \cite[pp.4573, 4568]{Wi11} differs from ours by a factor $v^{\ell(w_\circ^J)}$.

In the following lemma we shall establish the precise relation between the two inner products \eqref{bilinear-form-2} and \eqref{eq:bform}. Recall the canonical basis element $\mathcal{C}_w\in \widetilde{\HH}$ in \eqref{eq:KL}.
\begin{lem}
\label{lem:coincidence of two bilinear forms}
For any $C, C'\in \Xi,$ we have $$(\{C\}, \{C'\})= \big\langle \mathcal{C}_{(w_{C}^{+})^{-1}}, \mathcal{C}_{(w_{C'}^{+})^{-1}} \big\rangle.$$
\end{lem}

\begin{proof}
Recall the notation $H_C$ from \eqref{eq:HC}.

Similar to the proof of \cite[Proposition 9.7]{DDPW} (which actually makes sense for any type), for $A=(\gamma,w_{A},\mu)$ and $B=(\mu,w_{B},\nu)$, we have
\begin{align}\label{multiplication-form-1}
H_A H_B = v^{\ell(w_\circ^{\mu})}\pi(J_\mu) \sum\limits_{C} e_{A,B}^{C}H_C,
\end{align}
where $C=(\gamma,z,\nu)$ runs over $z\in\D_{\gamma\nu}$,
and $e_{A,B}^{C}\in \A.$ Moreover, we can write $H_B =x_{\mu}h$ for some $h\in \widetilde{\HH}.$ Thanks to $H_A H_B =\pi(J_\mu) H_A h$ and \eqref{eq:phig}, we compute
\begin{align*}
\phi_{\lambda\mu}^{w_{A}}\phi_{\mu\nu}^{w_{B}}(x_{\nu})=&
\phi_{\lambda\mu}^{w_{A}}(v^{\ell(w_\circ^\nu)}H_B)=v^{\ell(w_\circ^\nu)}\phi_{\lambda\mu}^{w_{A}}(x_{\mu})h
=v^{\ell(w_\circ^\nu)+\ell(w_\circ^\mu)}H_A h \\
=& v^{\ell(w_\circ^\nu)+2\ell(w_\circ^\mu)}\sum\limits_{C}e_{A,B}^{C}H_C
= v^{2\ell(w_\circ^\mu)}\sum_{C=(\lambda,z,\nu)}e_{A,B}^{C}\phi_{\lambda\nu}^{z}(x_{\nu}).
\end{align*}
Hence,  by \eqref{eq:Sbasis} and \eqref{eq:fABC} we have
\begin{align}
 \label{eq:fe}
 f_{A, B}^{C}=v^{2\ell(w_\circ^\mu)}v^{d_{A}+d_{B}-d_{C}}e_{A,B}^{C}.
\end{align}
Hence, for $C=(\gamma, w_{C}, \nu)\in \Xi,$ we have by \eqref{bilinear-form-2} and \eqref{eq:fe} that
\begin{align}\label{bilinear-form-add4}
([C], [C])=v^{2\ell(w_\circ^\gamma)}v^{-d_{C}+d_{C^{t}}}v^{d_{C}+d_{C^{t}}}e_{C^{t}, C}^{(\nu,\id,\nu)}
=v^{2\ell(w_\circ^\gamma) + 2d_{C^{t}}}e_{C^{t}, C}^{(\nu,\id,\nu)}.
\end{align}

On the other hand, using the notation in \cite[pp.4569]{Wi11}, we set
\begin{align*}
{}^{J_\gamma}\!H_{C}^{J_\nu} :=v^{\ell(w_{C}^{+})}H_C.
\end{align*}
By the definition \eqref{eq:bform} of the bilinear form $\langle \cdot, \cdot\rangle$ and \eqref{multiplication-form-1}, we have
\begin{align}
\big\langle{}^{J_\gamma}\!H_{C}^{J_\nu}, {}^{J_\gamma}\!H_{C}^{J_\nu} \big\rangle
=v^{2\ell(w_{C}^{+})}e_{C, C^{t}}^{(\gamma,\id,\gamma)},
\qquad
\big \langle{}^{J_\nu}\!H_{C^t}^{J_\gamma}, {}^{J_\nu}\!H_{C^t}^{J_\gamma} \big \rangle
=v^{2\ell(w_{C}^{+})}e_{C^{t}, C}^{(\nu,\id,\nu)}.
\label{bilinear-form-add5}
\end{align}
Note that $d_{C^{t}}=\ell(w_{C}^{+})-\ell(w_\circ^\gamma)$ by \eqref{eq:dC}. Therefore, by \eqref{bilinear-form-add4}--\eqref{bilinear-form-add5} we have
\begin{align}\label{bilinear-form-add6}
([C], [C])= \big\langle{}^{J_\nu}\!H_{C^t}^{J_\gamma}, {}^{J_\nu}\!H_{C^t}^{J_\gamma} \big\rangle.
\end{align}
Moreover, we have (see \cite[Lemma 2.13]{Wi11})
\begin{align}
  \label{eq:CC1}
([C], [C'])=0 =\big\langle{}^{J_\nu}\!H_{C^t}^{J_\gamma}, {}^{J_\nu}\!H_{(C')^t}^{J_\gamma} \big\rangle,
\qquad \text{ for }C\neq C'.
\end{align}

By \cite{Cur85} (also cf. \cite[Lemma~1.5]{Du92}) we have
\begin{align}
  \label{eq:CBstd}
\mathcal{C}_{(w_{C}^{+})^{-1}}={}^{J_\nu}\!H_{C^t}^{J_\gamma}+\sum\limits_{\tilde{C}< C^{t}}P_{\widetilde{C}, C^{t}}{}^{J_{\nu}}\!H_{\widetilde{C}}^{J_{\gamma}}.
\end{align}
The lemma now follows by \eqref{eq:A} and \eqref{bilinear-form-add6}--\eqref{eq:CBstd}.
\end{proof}

\begin{thm}
\label{thm:inner}
For any $C, C'\in \Xi,$ we have $(\{C\}, \{C'\})\in \delta_{C, C'}+v\mathbb{N}[v].$
\end{thm}

\begin{proof}
According to \cite[pp.9]{Wi08} and \cite[Theorem 5.4.1]{Wi08}, the fully-faithfulness of hypercohomology implies that the geometric form (using perverse sheaves) and the algebraic form (using Soergel bimodules) agree. (We remark that the shift $[-\ell(w_J)]$ in \cite[Theorem 5.4.1]{Wi08} disappears when using our normalized inner product \eqref{eq:bform}.) Therefore, we have $\big\langle \mathcal{C}_{(w_{C}^{+})^{-1}}, \mathcal{C}_{(w_{C'}^{+})^{-1}} \big\rangle \in \delta_{C, C'}+v\mathbb{N}[v].$ The lemma now follows from this and the identification of inner products in Lemma~ \ref{lem:coincidence of two bilinear forms}.
\end{proof}

\section{Cells in affine $q$-Schur algebras}
\label{sec:cells}

We shall identify the affine $q$-Schur algebras $\mathbf{S}'_v \equiv \mathbf{S}_v$ and their canonical bases via the isomorphism $\aleph$ in Proposition~\ref{addprop:alg-isomor}, and use only the notation $\mathbf{S}_v$ in this section. We shall provide a classification of left, right and two-sided cells in $\Sc_v$ and study their properties, which is inspired by \cite{Mc03}. The positivity of the canonical basis for $\mathbf{S}_v$ will be used crucially which originates from the geometric construction.

\subsection{The $\mathfrak{a}$-function}

Let $h_{x,y}^{z}$ denote the structure constants of $\widetilde{\HH}$ with respect to the KL canonical basis $\{\mathcal{C}_{w}\:|\:w\in \widetilde{W}\}$, that is,
\[
\mathcal{C}_{x}\cdot\mathcal{C}_{y}=\sum_{z\in \widetilde{W}}h_{x,y}^{z}\mathcal{C}_{z}.
\]
Recall from \eqref{eq:positivity} that $g_{A, B}^{C}$ are the structure constants of $\Sc_v$ with respect to the canonical basis $\B (\Sc_v)$.
Recall $J_\gamma$ from \eqref{Jgamma}, and $\pi(J)$ from \eqref{eq:piJ}. The following lemma establishes a relation between the above two structure constants, whose proof is similar to the one for finite type A (cf. \cite[Proposition~ 3.4]{Du92}).

\begin{lem}
\label{lem:d=d1=aaabc}
For any $A=(\gamma, w_A, \mu), B=(\mu, w_B, \nu), C=(\gamma, w_C, \nu)\in \Xi$, we have
\[
\pi(J_{\mu}) g_{A, B}^{C}=h_{w_{A}^{+}, w_{B}^{+}}^{w_{C}^{+}}.
\]
\end{lem}

\begin{proof}
Since $\mathcal{C}_{w_\circ^\mu}$ is a $q$-symmetrizer, we have $\mathcal{C}_{w_\circ^\mu}\mathcal{C}_{w_{B}^{+}}=\pi(J_{\mu})\mathcal{C}_{w_{B}^{+}}$.
We then compute
\begin{align*}
\sum_{C\in \Xi}\pi(J_{\mu})g_{A, B}^{C}\mathcal{C}_{w_{C}^{+}}
&=\{A\}\{B\}(\pi(J_{\mu})\mathcal{C}_{w_\circ^\nu})
=\{A\}(\pi(J_{\mu})\mathcal{C}_{w_{B}^{+}})\\
&=\{A\}(\mathcal{C}_{w_\circ^\mu}\mathcal{C}_{w_{B}^{+}})
=\{A\}(\mathcal{C}_{w_\circ^\mu})\mathcal{C}_{w_{B}^{+}}\\
&=\mathcal{C}_{w_{A}^{+}}\mathcal{C}_{w_{B}^{+}}
=\sum_{z\in W}h_{w_{A}^{+},w_{B}^{+}}^{z}\mathcal{C}_{z}.
\end{align*}
A comparison of the coefficients of $\mathcal{C}_{w_{C}^{+}}$ of both sides gives us the desired equality.
\end{proof}

The $\mathbf{a}$-function on $\widetilde{W}$ introduced by Lusztig (cf. \cite[\S13.6]{Lu03}) is defined as follows:
\begin{equation*}
\mathbf{a}: \widetilde{W}\longrightarrow \mathbb{N}, \quad \mathbf{a}(z)=\min\{k\in\mathbb{N}~|~v^kh_{x,y}^z\in\mathbb{Z}[v]\ \mbox{for all $x,y\in\widetilde{W}$}\}.
\end{equation*}
For any $x,y,z\in\widetilde{W}$, denote by $\gamma_{x,y}^{z^{-1}} \in \N$ the unique integer such that
$$h_{x,y}^{z} -\gamma_{x,y}^{z^{-1}}v^{-\mathbf{a}(z)} \in v^{-\mathbf{a}(z)+1}\mathbb{Z}[v].$$
It is clear by definition of the $\mathbf{a}$-function that for each $z\in \widetilde{W}$ there exist $x,y\in\widetilde{W}$ satisfying $\gamma_{x,y}^{z^{-1}}\neq0$.

For $A,B,C \in \Xi$, let $n_{A, B}^{C}$ be the minimal nonnegative integer $n$ such that $v^{n}g_{A, B}^{C}\in \mathbb{Z}[v]$.
Note that $\{n_{A, B}^{C}~|~A,B,C\in\Xi\}$ is a finite set because of Lemma \ref{lem:d=d1=aaabc} and the fact (cf. \cite[Theorem~7.2]{Lu85b}) that
 $v^{\ell(w_0)}h_{x,y}^z\in\mathbb{Z}[v]$ for any $x,y,z\in\widetilde{W}$, where $w_0$ is the longest element in the finite Weyl group $W$.
We define the \emph{$\mathfrak{a}$-function} $$\mathfrak{a}: \Xi \rightarrow \mathbb{Z},$$ which can be viewed as a $q$-Schur algebra analog of the $\mathbf{a}$-function, as follows:
\begin{equation}
\label{def:afunction}
\mathfrak{a}(C)=\max\left\{n_{A, B}^{C}+\ell(w_\circ^\mu)-\ell(w_\circ^\nu)\middle|
\begin{array}{c}
A, B\in\Xi, \mu\in\Lambda_{\ff},\\
\ro(B)=\mu, \co(B)=\nu=\co(C)
\end{array}
\right\}.
\end{equation}
We remark that the $\mathfrak{a}$-function is well defined since for each $C\in\Xi$ the set on the RHS of \eqref{def:afunction} is finite, thanks to the finiteness of $\Lambda_{\ff}$ and $\{n_{A, B}^{C}~|~A,B,C\in\Xi\}$.


It follows by definition of the $\mathfrak{a}$-function that $v^{\mathfrak{a}(C)+\ell(w_\circ^\nu)-\ell(w_\circ^\mu)} g_{A, B}^{C} \in \Z[v]$, for $B=(\mu, w_{B}, \nu)$. Recall $C^t$ from \eqref{eq:Ct}. We set
\[
\gamma_{A, B}^{C^{t}} \in \Z \text{ such that } v^{\mathfrak{a}(C)+\ell(w_\circ^\nu)-\ell(w_\circ^\mu)} g_{A, B}^{C} -\gamma_{A, B}^{C^{t}} \in v\Z[v].
\]
By definition of the $\mathfrak{a}$-function, for each $C\in\Xi$, there exist $A,B\in\Xi$ such that $\gamma_{A, B}^{C^{t}} \neq 0$.

\subsection{Distinguished elements in $\Sc_v$}


We define $\Delta(z) \in \N$ and $n_z \in \Z\setminus\{0\}$ to be the unique integers (cf. \cite[\S14.1]{Lu03}) such that
\begin{align*}
p_{\id,z} - n_z  v^{\Delta(z)}  \in v^{\Delta(z) +1} \Z[v].
\end{align*}
An element $z \in \widetilde{W}$ is called {\em distinguished} if $\mathbf{a}(z)=\Delta(z)$.

Note that $\{(\nu,\id,\nu) \} =[(\nu,\id,\nu)]$, and that, by Theorem~\ref{thm:inner}, $(\{C\}, [(\nu,\id,\nu)]) \in \N[v]$, for $\nu\in \Lambda_{\ff}$.
For $C \in \Xi$ with $\ro(C)=\co(C)$, we define $\Delta(C) \in \N$ and $n_C \in \Z_{>0}$ to be the unique integers such that
\begin{align}
 \label{eq:nC}
(\{C\}, [(\ro(C),\id,\ro(C))]) - n_C\, v^{\Delta(C)}  \in v^{\Delta(C) +1} \Z[v].
\end{align}

\begin{lem}
\label{lem:d=d1-ad-add1}
For any $C \in \Xi$ with $\ro(C)=\co(C)$, we have $\mathfrak{a}(C)\leq\Delta(C).$
\end{lem}

\begin{proof}
Let $C=(\nu, w_{C}, \nu)$.
Choose $A=(\nu, w_{A},\mu)$ and $B=(\mu, w_{B},\nu)$, for $\mu\in \Lambda_{\ff}$, such that
\begin{align}
 \label{eq:g1}
g_{A, B}^{C}=\gamma_{A, B}^{C^{t}}v^{-\mathfrak{a}(C) -\ell(w_\circ^\nu) +\ell(w_\circ^\mu)}+\mathrm{higher~degree~terms}, \text{ and } \gamma_{A, B}^{C^{t}}\neq 0.
\end{align}
We have the following inner product
\begin{align}\label{inner-product-1}
(\{A\}\{B\}, [(\nu,\id,\nu)])=\sum\limits_{E\in \Xi}g_{A, B}^{E}(\{E\}, [(\nu,\id,\nu)]).
\end{align}
By Lemmas \ref{addlem:anti-automorphismrho}--\ref{lem:d=d00-aa} and Theorem~\ref{thm:inner}, we have
\[
(\{A\}\{B\}, [(\nu,\id,\nu)]) =v^{d_{A^{t}} -d_A}(\{B\}, \{A^{t}\}) \in v^{d_{A^{t}} -d_A}\mathbb{N}[v].
\]
By Proposition \ref{prop:positive} and Theorem~\ref{thm:inner} again, all the summands on the RHS of \eqref{inner-product-1} have nonnegative integer coefficients. In particular, we must have
\begin{align}
  \label{eq:ginner1}
g_{A, B}^{C}(\{C\}, [(\nu,\id,\nu)])\in v^{d_{A^{t}} -d_A}\mathbb{N}[v].
\end{align}
On the other hand, since $d_{A^t}-d_{A}= \ell(w_\circ^\mu) -\ell(w_\circ^\nu)$ by \eqref{eq:dC}, it follows by \eqref{eq:nC}--\eqref{eq:g1} that
\begin{align}
 \label{eq:ginner2}
g_{A, B}^{C}(\{C\}, [(\nu,\id,\nu)]) =n_{C}\gamma_{A, B}^{C^{t}}v^{\Delta(C) -\mathfrak{a}(C)+d_{A^t}-d_{A}} +\mathrm{higher~terms},
\end{align}
where $n_{C} \gamma_{A, B}^{C^{t}}\neq 0.$
The desired inequality follows now by comparing \eqref{eq:ginner1}--\eqref{eq:ginner2}.
\end{proof}

\begin{definition}
An element $D\in\Xi$ is called {\em distinguished} if $\ro(D) =\co(D)$ and $\mathfrak{a}(D)=\Delta(D).$ Let $\mathcal{D}$ be the set of distinguished elements in $\Xi$.
\end{definition}


The following lemma gives a characterization of distinguished elements in $\Sc_v$.
\begin{lem}
\label{lem:d=d1}
An element $D\in\Xi$ is distinguished
if and only if $\ro(D)=\co(D)$ and $w_{D}^{+}\in\widetilde{W}$ is distinguished.
\end{lem}
\begin{proof}
We can write $D=(\nu,w_{D},\nu)$, for some $\nu\in \Lambda_{\ff}$.
By definition of the inner product and \eqref{eq:A}, we have $(\{D\}, [(\nu,\id,\nu)])=P_{(\nu,\id,\nu), D}$. We see that $P_{(\nu,\id,\nu), D}=p_{w_\circ^\nu, w_{D}^{+}}$ by \eqref{eq:p=P}.
By \cite[Theorem 6.6(c)]{Lu03}, we have $p_{w_\circ^\nu, w_{D}^{+}}=v^{-\ell(w_\circ^\nu)}p_{\id, w_{D}^{+}}.$ 
Hence, we have $(\{D\}, [(\nu,\id,\nu)])= v^{-\ell(w_\circ^\nu)}p_{\id, w_{D}^{+}}$, and by definitions of $\Delta$'s, we see
\begin{align}  \label{eq:DD}
\Delta(w_{D}^{+})=\Delta(D)+\ell(w_\circ^\nu).
\end{align}

$(\Longrightarrow)$ Suppose that $D$ is distinguished, i.e., $\mathfrak{a}(D)=\Delta(D)$.
By Lemma \ref{lem:d=d1=aaabc}, we have $\mathfrak{a}(D)+\ell(w_\circ^\nu)\leq \mathbf{a}(w_{D}^{+})$. By \cite[\S15.2]{Lu03}, we have $\mathbf{a}(w_{D}^{+})\leq \Delta(w_{D}^{+})$.
Then $\mathbf{a}(w_{D}^{+})\leq \Delta(w_{D}^{+}) =\Delta(D) +\ell(w_\circ^\nu) =\mathfrak{a}(D)+\ell(w_\circ^\nu) \leq \mathbf{a}(w_{D}^{+})$.
Hence $\mathbf{a}(w_{D}^{+})=\Delta(w_{D}^{+}),$ i.e., $w_{D}^{+}$ is distinguished.

$(\Longleftarrow)$  Suppose that $w_{D}^{+}$ is distinguished, i.e., $\mathbf{a}(w_{D}^{+})=\Delta(w_{D}^{+}).$ By Lemma \ref{lem:d=d1=aaabc}, we have $\mathfrak{a}(D)+\ell(w_\circ^\nu)\leq \mathbf{a}(w_{D}^{+}).$ By \cite[Conj.~ 14.2(P13) and \S15.6]{Lu03} we have $\gamma_{(w_{D}^{+})^{-1}, w_{D}^{+}}^{(w_{D}^{+})^{-1}}\neq 0,$ that is, the structure constant $h_{(w_{D}^{+})^{-1}, w_{D}^{+}}^{w_{D}^{+}}$ has $-\mathbf{a}(w_{D}^{+})$ as the lowest power of $v$, and so $\mathbf{a}(w_{D}^{+})\leq \mathfrak{a}(D)+\ell(w_\circ^\nu)$ by Lemma \ref{lem:d=d1=aaabc} again. Hence
\begin{align}  \label{eq:aa}
 \mathbf{a}(w_{D}^{+}) =\mathfrak{a}(D)+\ell(w_\circ^\nu).
\end{align}
By assumption $\mathbf{a}(w_{D}^{+})=\Delta(w_{D}^{+})$ and \eqref{eq:DD}--\eqref{eq:aa}, we conclude that $\mathfrak{a}(D)=\Delta(D)$, i.e., $D$ is distinguished.
\end{proof}

\begin{cor}
\label{cor:DDt}
If $D\in\Xi$ is distinguished then so is $D^t$.
\end{cor}
\begin{proof}
It follows from Lemma~\ref{lem:d=d1} and the fact (cf. \cite[\S14.1]{Lu03}) that if $z\in W$ is distinguished then so is $z^{-1}$.
\end{proof}

\subsection{Cells in $\Sc_v$}

For $C, C'\in \Xi$, we write $C\preceq_{L}C'$ (resp. $C\preceq_{R}C'$) if the canonical basis element $\{C\}$ appears with nonzero coefficient in the canonical basis expansion of $h\{C'\}$ (resp. $\{C'\}h$), for some $h \in \B (\Sc_v)$. Write $C\preceq_{LR}C'$ if $\{C\}$ appears with nonzero coefficient in the canonical basis expansion of $h\{C'\}h'$, for some $h, h' \in \B (\Sc_v)$. These relations are pre-orders, and they induce the corresponding equivalence relations $\sim_{L}$, $\sim_{R}$ and $\sim_{LR}.$ We call the resulting equivalence classes the left, right and two-sided cells of $\Xi$, respectively.  Alternatively, a cell can be understood as consisting of the canonical basis elements $\{C\}$ in $\Sc_v$ parameterized by elements $C$ in a cell of $\Xi$.

The same notions in the same notations above were defined for the extended affine Weyl group $\widetilde{W}$ or $\widetilde{\HH}$ with respect to the KL canonical basis $\{\mathcal{C}_{w}\:|\:w\in \widetilde{W}\}$, cf. \cite{KL79, Lu03}.

The following proposition provides a classification of left, right and two-sided cells in $\Sc_v.$
\begin{prop}
\label{prop:d=d2}
Let $C, C'\in \Xi$. The following hold:
\begin{enumerate}
\item
$C\preceq_{L}C'$ if and only if $\co(C)=\co(C')$ and $w_{C}^{+}\preceq_{L} w_{C'}^{+}$. Similarly, $C\preceq_{R}C'$ if and only if $\ro(C)=\ro(C')$ and $w_{C}^{+}\preceq_{R} w_{C'}^{+}$.
\item
$C\sim_{L}C'$ if and only if $\co(C)=\co(C')$ and $w_{C}^{+}\sim_{L} w_{C'}^{+}$. Similarly, $C\sim_{R}C'$ if and only if $\ro(C)=\ro(C')$ and $w_{C}^{+}\sim_{R} w_{C'}^{+}$.
\item
$C\sim_{LR}C'$ if and only if $w_{C}^{+}\sim_{LR} w_{C'}^{+}$.
\item
$\mathfrak{a}(C)=\mathbf{a}(w_{C}^{+})-\ell(w_\circ^\nu)$ where $\nu=\co(C)$.
\end{enumerate}
\end{prop}
\begin{proof}
(1) The proof is similar to that of \cite[Lemma 2.2]{Du96}. The ``only if'' part follows from Lemma \ref{lem:d=d1=aaabc}.

Conversely, suppose that $\co(C) =\co(C') =\gamma$ and $w_{C}^{+}\preceq_{L} w_{C'}^{+}$. Then $\mathcal{C}_{w_{C}^{+}}$ appears with nonzero coefficient in the product $\mathcal{C}_{w}\mathcal{C}_{w_{C'}^{+}}$ for some $w$. Denote $C=(\lambda,w_{C},\gamma)$ and $C'=(\mu,w_{C'},\gamma)$, for $\lambda,\mu\in \Lambda_{\ff}$. By positivity of the structure constants with respect to the Kazhdan-Lusztig basis and \cite[Theorem 6.6(b)]{Lu03}, $\mathcal{C}_{w_{C}^{+}}$ appears with nonzero coefficient in $\mathcal{C}_{w_\circ^\lambda}\mathcal{C}_{w}\mathcal{C}_{w_\circ^\mu}\mathcal{C}_{w_{C'}^{+}}$. By \cite[(1.9)-(1.10)]{Cur85} $\mathcal{C}_{w_\circ^\lambda}\mathcal{C}_{w}\mathcal{C}_{w_\circ^\mu}$ is a linear combination of the elements $\mathcal{C}_{w_{D}^{+}}$ ($D=(\lambda,w_D,\mu)\in\Xi$). Thus, $\mathcal{C}_{w_{C}^{+}}$ appears with nonzero coefficient in some product $\mathcal{C}_{w_{D}^{+}}\mathcal{C}_{w_{C'}^{+}}.$ It follows by Lemma \ref{lem:d=d1=aaabc} and the hypothesis that $\{C\}$ appears with nonzero coefficient in $\{D\}\{C'\}$, hence $C\preceq_{L}C'$.

The proof for the second claim on $\preceq_{R}$ is entirely similar.

(2) follows from (1).

(3) The ``only if'' part follows from Lemma \ref{lem:d=d1=aaabc}.

Conversely, suppose that $w_{C}^{+}\sim_{LR} w_{C'}^{+}$. By \cite[\S3.1(k)(l)]{Lu87c} there exists $x\in \widetilde{W}$ such that $w_{C}^{+}\sim_{L} x \sim_{R}w_{C'}^{+}$. Denote $\co(C)=\gamma$ and $\ro(C') =\mu$.
Then by \cite[Proposition 2.4]{KL79} we have $x=w_{D}^{+},$ for $D=(\mu,w_D,\gamma)\in\Xi$. Therefore, by (2) we have $C\sim_{L}D\sim_{R}C'$, hence $C\sim_{LR}C'$.

(4)  By \cite[Conj.~ 14.2(P13) and \S15.6]{Lu03} there exists a unique distinguished element $d$ in the left cell containing $w_{C}^{+}$ such that $\gamma_{(w_{C}^{+})^{-1},w_{C}^{+}}^{d}\neq 0$. By \cite[Proposition~ 2.4]{KL79} and \cite[Conj.~ 14.2(P6) and \S15.6]{Lu03}, this element $d$ is the longest one in a double coset of $W_{\nu}\backslash \widetilde{W}/W_{\nu}.$ By Lemma \ref{lem:d=d1=aaabc}, we have $\pi(J_{\nu}) g_{C, D}^{C}=h_{w_{C}^{+},d}^{w_{C}^{+}}$, for $D =(\nu,w_D,\nu)$ with $w_{D}^{+}=d$. By \cite[Conj.~ 14.2(P7) and \S15.6]{Lu03} we have $\gamma_{w_{C}^{+},d}^{(w_{C}^{+})^{-1}}=\gamma_{(w_{C}^{+})^{-1},w_{C}^{+}}^{d} \neq 0$. Hence $\mathbf{a}(w_{C}^{+})\leq \mathfrak{a}(C)+\ell(w_\circ^\nu).$

On the other hand, by Lemma~\ref{lem:d=d1=aaabc} we have $\mathfrak{a}(C)+\ell(w_\circ^\nu)\leq \mathbf{a}(w_{C}^{+})$. Thus (4) is proved.
\end{proof}

\begin{cor}
\label{lem:d=d3}
\begin{enumerate}
\item
There are finitely many two-sided cells in $\Sc_v$.
\item
Suppose $\Lambda_{\ff}$ contains at least one regular $\widetilde{W}$-orbit. Then there is a one-to-one correspondence between the two-sided cells in $\Sc_v$ and the two-sided cells in  $\widetilde{\HH}$.
\end{enumerate}
\end{cor}

\begin{proof}
It has been known (cf. \cite[Theorem~2.2(a)]{Lu87a}) that there are finitely many two-sided cells in the affine Hecke algebra $\widetilde{\HH}$. Hence the corollary follows from this and Proposition \ref{prop:d=d2}(3).
\end{proof}

\subsection{Properties of cells in $\Sc_v$}

The following theorem is a $q$-Schur algebra analog of properties for cells in Hecke algebras (cf. \cite[Conj.~ 14.2(P1)--(P15)]{Lu03}, which holds for finite/affine Hecke algebras of equal parameters  \cite[\S15]{Lu03}). (Note that a counterpart is not formulated here for (P12) of \cite[Conj.~ 14.2]{Lu03} which compares the $\mathbf{a}$-functions for a Coxeter group and its parabolic subgroup.)

\begin{thm}
  \label{thm:cell structures}
The following properties hold.

\begin{itemize}
\item[(P1)] For any $C \in \Xi$ with $\ro(C)=\co(C)$, we have $\mathfrak{a}(C)\leq\Delta(C)$.
\item[(P2)] If $D\in \mathcal{D}$ and $\gamma_{A, B}^{D}\neq 0$, for $A, B \in \Xi$, then $B=A^{t}$.
\item[(P3)] For each $C\in \Xi$, there is a unique $D\in \mathcal{D}$ such that $\gamma_{C, C^{t}}^{D}\neq 0$.
\item[(P4)] If $C\preceq_{LR}C'$ and $\co(C) =\co(C')$,
then $\mathfrak{a}(C)\geq \mathfrak{a}(C')$. Hence, if $C\sim_{LR}C'$ and  $\co(C) =\co(C')$,
then $\mathfrak{a}(C)=\mathfrak{a}(C')$.
\item[(P5)] If $D\in \mathcal{D}$ and $\gamma_{A^{t}, A}^{D}\neq 0$ for $A\in \Xi$, then $\gamma_{A^{t}, A}^{D}=n_{D}=1$.
\item[(P6)] If $D\in \mathcal{D}$, then $D=D^{t}$.
\item[(P7)] We have $\gamma_{A, B}^{C}=\gamma_{C, A}^{B}=\gamma_{B, C}^{A}$, for any $A, B, C \in \Xi$.
\item[(P8)] Suppose that $\gamma_{A, B}^{C}\neq 0$, for $A, B, C \in \Xi$. Then $A\sim_{L}B^{t}$, $B\sim_{L}C^{t}$ and $C\sim_{L}A^{t}$.
\item[(P9)] If $C\preceq_{L}C'$ and $\mathfrak{a}(C)=\mathfrak{a}(C')$, then $C\sim_{L}C'$.
\item[(P10)] If $C\preceq_{R}C'$ and $\mathfrak{a}(C^{t})=\mathfrak{a}(C'^{t})$, then $C\sim_{R}C'$.
\item[(P11)] If $C\preceq_{LR}C'$, $\co(C) =\co(C')$ and $\mathfrak{a}(C)=\mathfrak{a}(C'),$ then $C\sim_{LR}C'$.
%
\item[(P13)] Each left cell $\Gamma$ in $\Sc_v$ contains a unique distinguished element $D$. Moreover, we have $\gamma_{C^{t}, C}^{D}\neq 0$, for all $C\in \Gamma$.
\item[(P14)] We have $C\sim_{LR}C^{t}$, for any $C\in \Xi$.
\item[(P15)] Let $v'$ be a second indeterminate and let $\hat{g}_{A, B}^{C}\in \mathbb{Z}[v', v'^{-1}]$ be obtained from $g_{A, B}^{C}$ by the substitution $v\mapsto v'$. If $C, C', A, B\in \Xi$ with $A, B$ belonging to the same two-sided cell (saying $\mathbf{c}$), then
\[
\sum\limits_{A'\in \mathbf{c}}\hat{g}_{B, C'}^{A'}g_{C, A'}^{A}=\sum\limits_{A'\in \mathbf{c}}g_{C, B}^{A'}\hat{g}_{A', C'}^{A}.
\]
\end{itemize}
\end{thm}

\begin{proof}
(P1). It is Lemma \ref{lem:d=d1-ad-add1}.
\vspace{1mm}

(P2).
Let $\co(D)=\nu$. Recalling \eqref{eq:positivity}, we have
\begin{align}\label{inner-product-1-add}
(\{A\}\{B\}, [(\nu,\id,\nu)])=\sum\limits_{E\in \Xi} g_{A, B}^{E}(\{E\}, [(\nu,\id,\nu)]).
\end{align}
By Corollary~\ref{cor:DDt}, we have $D^{t}\in \mathcal{D}$ and thus
$\Delta(D^{t}) =\mathfrak{a}(D^{t}).$
The identity \eqref{eq:ginner2} in our setting reads
\begin{align}
 \label{eq:ginner2-add}
g_{A, B}^{D^{t}}(\{D^{t}\}, [(\nu,\id,\nu)]) =n_{D^{t}}\gamma_{A, B}^{D}v^{d_{A^t}-d_{A}} +\mathrm{higher~terms},
\end{align}
where $n_{D^{t}} \gamma_{A, B}^{D}\neq 0$.

By Proposition \ref{prop:positive} and Theorem~\ref{thm:inner}, all the summands on the RHS of \eqref{inner-product-1-add} have nonnegative integer coefficients. Hence by \eqref{eq:ginner2-add} the coefficient of $v^{d_{A^t}-d_{A}}$ on the RHS of \eqref{inner-product-1-add} is positive.
By Lemmas \ref{addlem:anti-automorphismrho}--\ref{lem:d=d00-aa}, we have that
  the LHS of \eqref{inner-product-1-add} equals to $v^{d_{A^{t}} -d_A}(\{B\}, \{A^{t}\})$. By Theorem~\ref{thm:inner}, we must have $B=A^{t}$.
\vspace{1mm}

(P3).
Let $\ro(C)=\nu$. We have by \eqref{eq:positivity} that
\begin{align}\label{eq:BF1}
(\{C\}\{C^{t}\}, [(\nu,\id,\nu)])=\sum\limits_{E\in \Xi}g_{C, C^{t}}^{E}(\{E\}, [(\nu,\id,\nu)]).
\end{align}
All the summands on the RHS of \eqref{eq:BF1} have nonnegative integer coefficients, by Proposition ~\ref{prop:positive} and Theorem~\ref{thm:inner}. The identity \eqref{eq:ginner2} in our setting reads
\begin{align}
 \label{eq:ginner2-add-Newadd}
g_{C, C^{t}}^{E}(\{E\}, [(\nu,\id,\nu)]) =n_{E}\gamma_{C, C^{t}}^{E^{t}} v^{\Delta(E) -\mathfrak{a}(E)+d_{C^t}-d_{C}} +\mathrm{higher~terms},
\end{align}
with $n_{E}> 0$ and $\Delta(E) -\mathfrak{a}(E) \ge 0$.

On the other hand, by Lemmas \ref{addlem:anti-automorphismrho}--\ref{lem:d=d00-aa} and Theorem~\ref{thm:inner}, we have
\begin{align}
  \label{eq:inner1}
(\{C\}\{C^{t}\}, [(\nu,\id,\nu)]) =v^{d_{C^{t}} -d_C}(\{C^{t}\}, \{C^{t}\})\in v^{d_{C^{t}} -d_C} (1+v\mathbb{N}[v]).
\end{align}

A comparison of \eqref{eq:BF1}, \eqref{eq:ginner2-add-Newadd} and \eqref{eq:inner1} implies that there is a unique element $D\in \Xi$ such that
\begin{equation}
\label{eq:P3}
n_{D^t} \gamma_{C, C^{t}}^{D} =1,
\end{equation}
and $\Delta(D^{t})-\mathfrak{a}(D^{t})=0$, i.e., $D^{t}\in \mathcal{D}$.
In particular, $\gamma_{C, C^{t}}^{D}\neq 0$. Thanks to Corollary~\ref{cor:DDt}, we have $D\in \mathcal{D}$.
\vspace{1mm}

(P4).
By Lemma \ref{lem:d=d1=aaabc}, we see that $C\preceq_{LR}C'$ implies $w_{C}^{+}\preceq_{LR} w_{C'}^{+}$. Hence $\mathbf{a}(w_{C}^{+})\geq \mathbf{a}(w_{C'}^{+})$ by \cite[Conj.~ 14.2(P4) and \S15.6]{Lu03}. Therefore $\mathfrak{a}(C)\geq \mathfrak{a}(C')$ by Proposition \ref{prop:d=d2}(4). Then the second statement follows.
\vspace{1mm}

(P6).
(We prove (P6) before (P5).)
We have $\gamma_{A, A^t}^{D}\neq 0$, for some $A\in \Xi$, by (P2). Since the product $\{A\}\{A^{t}\}$ is preserved by the anti-automorphism $\Psi$ in \eqref{eq:psi}, we have $0\neq g_{A, A^{t}}^{D}=g_{A, A^{t}}^{D^{t}}$. By Proposition \ref{prop:d=d2}(4) and \cite[Proposition 13.9(a)]{Lu03}, we have $\mathfrak{a}(D)=\mathfrak{a}(D^{t}).$ Hence $0\neq \gamma_{A, A^{t}}^{D^{t}}=\gamma_{A, A^{t}}^{D}$. Since $D$ and $D^{t}$ are both distinguished, we have $D=D^{t}$ by (P3).
\vspace{1mm}

(P5). 
By \eqref{eq:P3} and (P6), we have
$n_{D}\gamma_{A^{t}, A}^{D}=n_{D^{t}}\gamma_{A^{t}, A}^{D}=1$.
Note that $n_{D}>0$ and $\gamma_{A^{t}, A}^{D}$ are both integers, which forces $n_{D}=\gamma_{A^{t}, A}^{D}=1$.
\vspace{1mm}

(P7).
We can assume that one of $\gamma_{A, B}^{C}, \gamma_{C, A}^{B},\gamma_{B, C}^{A}$ is nonzero, saying $\gamma_{A, B}^{C}\neq 0$. Then $g_{A, B}^{C^{t}}\neq 0.$ Hence we can write $A=(\lambda, w_{A}, \gamma),$ $B=(\gamma, w_{B}, \nu)$ and $C^{t}=(\lambda, (w_{C})^{-1}, \nu)$, for $\lambda, \gamma, \nu\in\Lambda_{\ff}.$ By Lemma \ref{lem:d=d1=aaabc} and Proposition \ref{prop:d=d2}(4), we have
\begin{align*}
\gamma_{A, B}^{C}=\gamma_{w_{A}^{+}, w_{B}^{+}}^{w_{C}^{+}},\quad \gamma_{C, A}^{B}=\gamma_{w_{C}^{+}, w_{A}^{+}}^{w_{B}^{+}} \quad\text{and}\quad \gamma_{B, C}^{A}=\gamma_{w_{B}^{+}, w_{C}^{+}}^{w_{A}^{+}}.
\end{align*}
By \cite[Conj.~ 14.2(P7) and \S15.6]{Lu03} we know $\gamma_{w_{A}^{+}, w_{B}^{+}}^{w_{C}^{+}}=\gamma_{w_{C}^{+}, w_{A}^{+}}^{w_{B}^{+}}=\gamma_{w_{B}^{+}, w_{C}^{+}}^{w_{A}^{+}}$. Hence $\gamma_{A, B}^{C}=\gamma_{C, A}^{B}=\gamma_{B, C}^{A}$.
\vspace{1mm}

(P8).
Since $\gamma_{A, B}^{C}\neq 0,$ from the proof of (P7) we have $0\neq\gamma_{A, B}^{C}=\gamma_{w_{A}^{+}, w_{B}^{+}}^{w_{C}^{+}}$. By \cite[Conj.~14.2(P8) and \S15.6]{Lu03} we have $w_{A}^{+}\sim_{L} (w_{B}^{+})^{-1}$, $w_{B}^{+}\sim_{L} (w_{C}^{+})^{-1}$ and $w_{C}^{+}\sim_{L} (w_{A}^{+})^{-1}$. Thus by Proposition~\ref{prop:d=d2}(2) we must have $A\sim_{L}B^{t}$, $B\sim_{L}C^{t}$ and $C\sim_{L}A^{t}$.
\vspace{1mm}

(P9).
Since $C\preceq_{L}C'$, we have $\co(C)=\co(C')$ and $w_{C}^{+}\preceq_{L} w_{C'}^{+}$ by Proposition~ \ref{prop:d=d2}(1). Since $\mathfrak{a}(C)=\mathfrak{a}(C'),$ by Proposition \ref{prop:d=d2}(4) we have $\mathbf{a}(w_{C}^{+})=\mathbf{a}(w_{C'}^{+})$, which implies that $w_{C}^{+}\sim_{L} w_{C'}^{+}$ by \cite[Conj.~ 14.2(P9) and \S15.6]{Lu03}. Hence we obtain $C\sim_{L}C'$ by Proposition~ \ref{prop:d=d2}(2).
\vspace{1mm}

(P10).
If $C\preceq_{R}C'$, we have $C^{t}\preceq_{L}C'^{t}$ through the anti-automorphism $\Psi$ in \eqref{eq:psi}. By (P9) we have $C^{t}\sim_{L}C'^{t}$, and hence $C\sim_{R}C'$.
\vspace{1mm}

(P11).
By Lemma~\ref{lem:d=d1=aaabc}, we see that $C\preceq_{LR}C'$ implies $w_{C}^{+}\preceq_{LR} w_{C'}^{+}$. Since $\co(C) =\co(C')$ and $\mathfrak{a}(C)=\mathfrak{a}(C')$, we have $\mathbf{a}(w_{C}^{+})=\mathbf{a}(w_{C'}^{+})$ by Proposition~\ref{prop:d=d2}(4). By \cite[Conj.~ 14.2(P11) and \S15.6]{Lu03} we have $w_{C}^{+}\sim_{LR} w_{C'}^{+}$. Hence we have $C\sim_{LR}C'$ by Proposition~\ref{prop:d=d2}(3).
\vspace{1mm}

(P13).
Take any $A\in \Gamma$ and let $\co(A) =\nu$.
By \cite[Conj.~ 14.2(P13) and \S15.6]{Lu03} there exists a unique distinguished element $d\in \widetilde{W}$ contained in the left cell containing $w_{A}^{+}$. Set $D =(\nu,w_D,\nu)$ with $w_{D}^{+}=d$, which is distinguished by Lemma~\ref{lem:d=d1}. Thanks to Proposition~\ref{prop:d=d2}(2), we have $D\sim_{L}A$ and hence $D\in \Gamma$. The uniqueness of $D$ is clear.

For each $C\in \Gamma$, we must have $\co(C)=\nu$. By Lemma~\ref{lem:d=d1=aaabc} and Proposition~\ref{prop:d=d2}(4), we have $\gamma_{C^{t}, C}^{D}=\gamma_{(w_{C}^{+})^{-1}, w_{C}^{+}}^{d}$. By Proposition~\ref{prop:d=d2}(2) we have $d\sim_{L} w_{C}^{+}$, and hence $\gamma_{(w_{C}^{+})^{-1}, w_{C}^{+}}^{d}\neq 0$ by \cite[Conj.~ 14.2(P13) and \S15.6]{Lu03}. Thus $\gamma_{C^{t}, C}^{D}\neq 0.$
\vspace{1mm}

(P14).
Note $w_{C^t}^{+} =(w_{C}^{+})^{-1}$.
By \cite[Conj.~ 14.2(P13) and \S15.6]{Lu03}, we have $w_{C}^{+}\sim_{LR}(w_{C}^{+})^{-1}$. Hence $C\sim_{LR}C^{t}$ by Proposition~\ref{prop:d=d2}(3).
\vspace{1mm}

(P15).
By Lemma \ref{lem:d=d1=aaabc}, we have either $g_{C, B}^{A'}=0$ or $g_{C, B}^{A'}=\pi(J_{\nu})^{-1}h_{w_{C}^{+}, w_{B}^{+}}^{w_{A'}^{+}}$, where $\co(C)=\nu$; we have $\hat{g}_{A', C'}^{A}=0$ or $\hat{g}_{A', C'}^{A}=\hat{\pi}(J_{\gamma})^{-1}\hat{h}_{w_{A'}^{+}, w_{C'}^{+}}^{w_{A}^{+}}$, where $\ro(C')=\gamma$, $\hat{\pi}(J_{\gamma})$ is obtained from $\pi(J_{\gamma})$ by the substitution $v\mapsto v'$, and $\hat{h}_{w_{A'}^{+}, w_{C'}^{+}}^{w_{A}^{+}}$ is obtained from $h_{w_{A'}^{+}, w_{C'}^{+}}^{w_{A}^{+}}$ by the substitution $v\mapsto v'$. Hence
\begin{align*}
\sum\limits_{A'\in \mathbf{c}}g_{C, B}^{A'}\hat{g}_{A', C'}^{A}=\sum\limits_{A'\in \mathbf{c}}\pi(J_{\nu})^{-1}\hat{\pi}(J_{\gamma})^{-1}h_{w_{C}^{+}, w_{B}^{+}}^{w_{A'}^{+}}\hat{h}_{w_{A'}^{+}, w_{C'}^{+}}^{w_{A}^{+}}.
\end{align*}
Similarly, we have
$$\sum\limits_{A'\in \mathbf{c}}\hat{g}_{B, C'}^{A'}g_{C, A'}^{A}=\sum\limits_{A'\in \mathbf{c}}\pi(J_{\nu})^{-1}\hat{\pi}(J_{\gamma})^{-1}\hat{h}_{w_{B}^{+}, w_{C'}^{+}}^{w_{A'}^{+}}h_{w_{C}^{+}, w_{A'}^{+}}^{w_{A}^{+}}.
$$
By Proposition~ \ref{prop:d=d2}(3) and \cite[Conj.~ 14.2(P15) and \S15.6]{Lu03} we have
\[
\sum\limits_{A'\in \mathbf{c}}h_{w_{C}^{+}, w_{B}^{+}}^{w_{A'}^{+}}\hat{h}_{w_{A'}^{+}, w_{C'}^{+}}^{w_{A}^{+}}=\sum\limits_{A'\in \mathbf{c}}\hat{h}_{w_{B}^{+}, w_{C'}^{+}}^{w_{A'}^{+}}h_{w_{C}^{+}, w_{A'}^{+}}^{w_{A}^{+}}.
\]
Hence (P15) follows from these identities.
\end{proof}

\begin{cor}\label{preceq-sim-siml}
Let $C, C'\in \Xi$. If $C\preceq_{L} C'$ and $C\sim_{LR}C'$, then $C\sim_{L}C'$.
\end{cor}

\begin{proof}
Since $C\preceq_{L}C'$, we have $\co(C) =\co(C')$. Since $C\sim_{LR}C'$, we have $\mathfrak{a}(C)=\mathfrak{a}(C')$ by Theorem \ref{thm:cell structures}(P4). Hence $C\sim_{L}C'$ by Theorem \ref{thm:cell structures}(P9).
\end{proof}

\section{Asymptotic Schur algebras}
\label{sec:asymp}
\subsection{The asymptotic Schur algebra $\mathbf{J}^{\mathbf S}$}

For a given two-sided cell $\mathfrak{c}$ in $\widetilde{W}$, let $\widetilde{\HH}_{\mathfrak{c}}$ be the $\mathbb{Z}[v]$-submodule of $\widetilde{\HH}$ generated by Kazhdan-Lusztig elements $\mathcal{C}_{y}$ ($y\in \mathfrak{c}$). For $w\in \mathfrak{c}$, let $\overline{\mathbf{a}}(w)$ denote the smallest nonnegative integer $n$ such that $v^{n}\mathcal{C}_{w}\subset \widetilde{\HH}_{\mathfrak{c}}$. It has been shown in \cite[\S2]{Lu95} that
\begin{equation}
\label{eq:abar=aH}
\mathbf{a}(w)=\overline{\mathbf{a}}(w)\quad \mbox{for any $w\in \widetilde{W}$}.
\end{equation}

Fix a two-sided cell $\mathbf{c}$ in $\Sc_v$. Let ${\mathbf S}_{v;\mathbf{c}}$ be the $\mathbb{Z}[v]$-submodule of $\Sc_v$ generated by the elements $\{\{C\}~|~C \in \mathbf{c}\}$.
For any $C\in \mathbf{c}$, if there exists $n\in \Z_{\geq0}$ such that $v^{n}\{C\}{\mathbf S}_{v;\mathbf{c}}\subset {\mathbf S}_{v;\mathbf{c}}$, then we set $\overline{\mathfrak{a}}(C)$ to be the smallest such $n$; otherwise, we set $\overline{\mathfrak{a}}(C)=\infty$.

The following lemma shows that the functions $\mathfrak{a}$ and $\overline{\mathfrak{a}}$ are the same, which will be used in the proof of Lemma~\ref{lem:P1}.
\begin{lem}
\label{lem:aC=}
For any $C\in\Xi$, we have $\overline{\mathfrak{a}}(C)=\overline{\mathbf{a}}(w_{C}^{+})-\ell(w_\circ^\nu)$, where $\nu=\co(C)$. Hence the function $\mathfrak{a}$ and $\overline{\mathfrak{a}}$ coincide, i.e.,
\begin{equation}
\label{eq:a=aS}
\mathfrak{a} (C) =\overline{\mathfrak{a}} (C), \quad(\forall C\in \Xi).
\end{equation}
\end{lem}
\begin{proof}
By \cite[Conj.~ 14.2(P13) and \S15.6]{Lu03} there exists a unique distinguished element $d$ in the left cell containing $w_{C}^{+}$ such that $\gamma_{(w_{C}^{+})^{-1},w_{C}^{+}}^{d}\neq 0.$ By \cite[Conj.~ 14.2(P7) and \S15.6]{Lu03} we have $\gamma_{w_{C}^{+},d}^{(w_{C}^{+})^{-1}}=\gamma_{(w_{C}^{+})^{-1},w_{C}^{+}}^{d} \neq 0.$ Recall
\begin{align*}
h_{w_{C}^{+},d}^{w_{C}^{+}}=\gamma_{w_{C}^{+},d}^{(w_{C}^{+})^{-1}}v^{-\overline{\mathbf{a}}(w_{C}^{+})}+\mathrm{higher~degree~terms}.
\end{align*}
Set $D =(\nu,d,\nu) \in \Xi$. By Lemma \ref{lem:d=d1=aaabc}, we have $p_{\nu}g_{C, D}^{C}=h_{w_{C}^{+},d}^{w_{C}^{+}}.$ By Proposition \ref{prop:d=d2}(3), $C$ and $D$ lie in the same two-sided cell. Hence we have $\overline{\mathbf{a}}(w_{C}^{+})\leq \overline{\mathfrak{a}}(C)+\ell(w_\circ^\nu).$

On the other hand, by Lemma \ref{lem:d=d1=aaabc} and Proposition \ref{prop:d=d2}(3), we have $\overline{\mathfrak{a}}(C)+\ell(w_\circ^\nu)\leq \overline{\mathbf{a}}(w_{C}^{+})$.
Therefore the first statement holds.
Then the second statement follows by Proposition~\ref{prop:d=d2}(4) and \eqref{eq:abar=aH}.
\end{proof}

For any two-sided cell $\mathbf{c}$ in $\Sc_v$ and $\nu\in \Lambda_{\ff}$, denote
$$\mathbf{c}[\nu]=\{C \in \mathbf{c}~|~\co(C)=\nu\}.$$
\begin{lem}
\label{lem:P1}
\begin{itemize}
\item[(1)] $\overline{\mathfrak{a}}(C)<\infty$ for any $C\in \Xi$.
\item[(2)]
The function $\overline{\mathfrak{a}}$ is constant on $\mathbf{c}[\nu]$ for any two-sided cell $\mathbf{c}$ in $\Sc_v$ and any $\nu\in \Lambda_{\ff}$.
\end{itemize}
\end{lem}

\begin{proof}The statements are clear by \eqref{eq:a=aS} and Theorem~\ref{thm:cell structures}(P4).
\end{proof}

Similar to \cite[\S1.4]{Lu95}, we can use Lemma \ref{lem:P1} to define a $\mathbb{Z}$-algebra $\mathbf{J}^{\mathbf S}_{\mathbf{c}}$ for each two-sided cell $\mathbf{c}$ in $\Sc_v$. This ring $\mathbf{J}^{\mathbf S}_{\mathbf{c}}$ has a $\mathbb{Z}$-basis $\{t_{C}\:|\:C\in \mathbf{c}\}$ with multiplication
\[
t_{A}t_{B}=\sum\limits_{C\in \mathbf{c}}\gamma_{A,B}^{C^{t}}t_{C}.
\]
Denote by $B_0$ the set of all two-sided cells in $\Sc_v$, which is finite by Corollary~\ref{lem:d=d3}(1). We shall denote by $\mathbf{J}^{\mathbf S}$ the direct sum of $\mathbb{Z}$-algebras $\oplus_{\mathbf{c}\in B_{0}}\mathbf{J}^{\mathbf S}_{\mathbf{c}}$, which is an associative $\mathbb{Z}$-algebra. We shall call $\mathbf{J}^{\mathbf S}$ the asymptotic Schur algebra.

It is known (cf. \cite[Theorem~2.2(b)]{Lu87a}) that the number of distinguished elements in $\widetilde{W}$ is finite. We note that the set $\Lambda_{\ff}$ is finite, too. Thus the set $\mathcal{D}$ and $\mathcal{D}_{\mathbf{c}}:=\mathcal{D}\cap\mathbf{c}$ are both finite by Lemma \ref{lem:d=d1}.

\begin{lem}
\label{lem:P2}
The element $\sum_{D\in \mathcal{D}_{\mathbf{c}}}t_{D}$ is the identity element of $\mathbf{J}^{\mathbf S}_{\mathbf{c}}$, while $\sum_{D\in \mathcal{D}}t_{D}$ is the identity element of $\mathbf{J}^{\mathbf S}$.
\end{lem}
\begin{proof}
Here we only provide the proof for the second statement.

For $A\in \Xi$, we have
\begin{align*}
(\sum_{D\in \mathcal{D}}t_{D})\cdot t_{A}&=\sum_{D\in \mathcal{D}}\sum_{B\in \Xi}\gamma_{D,A}^{B^{t}}t_{B}\\
&=\sum_{D\in \mathcal{D}}\sum_{B\in \Xi}\gamma_{A,B^{t}}^{D}t_{B} \qquad \text{ by Theorem~\ref{thm:cell structures}(P7)}\\
&=\sum_{D\in \mathcal{D}}\gamma_{A,A^{t}}^{D}t_{A} \qquad \text{ by Theorem~\ref{thm:cell structures}(P2)}\\
&=t_{A} \qquad \text{ by Theorem~\ref{thm:cell structures}(P3, P5)}.
\end{align*}
Similarly, we can verify $t_A\cdot(\sum_{D\in \mathcal{D}}t_{D})=t_A$. Hence $\sum_{D\in \mathcal{D}}t_{D}$ is the identity element of $\mathbf{J}^{\mathbf S}$.
\end{proof}

%

The following proposition gives a characterization of left, right and two-sided cells in $\Sc_v.$
\begin{prop}
\label{characterization-two-cells}
Let $C, C'\in \Xi$.
\begin{itemize}
\item[(1)] $C\sim_{L}C'$ if and only if $t_{C}t_{C'^{t}}\neq 0$; $C\sim_{L}C'$ if and only if $t_{C'}$ appears with nonzero coefficient in $t_{A}t_{C}$ for some $A\in \Xi$.
\item[(2)] $C\sim_{R}C'$ if and only if $t_{C^{t}}t_{C'}\neq 0$; $C\sim_{R}C'$ if and only if $t_{C'}$ appears with nonzero coefficient in $t_{C}t_{A}$ for some $A\in \Xi$.
\item[(3)] $C\sim_{LR}C'$ if and only if $t_{C}t_{A}t_{C'}\neq 0$ for some $A\in \Xi$; $C\sim_{LR}C'$ if and only if $t_{C'}$ appears with nonzero coefficient in $t_{A'}t_{C}t_{A}$ for some $A', A\in \Xi$.
\end{itemize}
\end{prop}
\begin{proof}
The proof is similar to that of \cite[Proposition 18.4]{Lu03} by using Theorem \ref{thm:cell structures} and the positivity of structure constants of $\Sc_v$ (cf. Proposition~\ref{prop:positive}).
\end{proof}

The following theorem is a $q$-Schur algebra analog of the algebra homomorphism for an Hecke algebra established in \cite[\S2.4]{Lu87a} (see \cite[Theorem 2.3]{Du95} and \cite[(4.9)]{Mc03} for finite and affine type $A$, respectively).

\begin{thm}
 \label{prop:d=d2=a'aa}
Let $\mathbf{c}$ be a two-sided cell in $\Sc_v$. Then the $\A$-module homomorphism $\Phi_{\mathbf{c}} :\Sc_v \rightarrow \A\otimes \mathbf{J}^{\mathbf S}_{\mathbf{c}}$ defined by
\[
\Phi_{\mathbf{c}}(\{C\})=\sum\limits_{D\in \mathcal{D}_{\mathbf{c}}, B\in \mathbf{c}}g_{C, D}^{B}t_{B}
\]
is an algebra homomorphism. The $\A$-module homomorphism $\Phi :\Sc_v\rightarrow \A\otimes \mathbf{J}^{\mathbf S}$ defined by
\begin{align}
\Phi (\{C\})&=\sum\limits_{\mathbf{c}\in B_{0}, D\in \mathcal{D}_{\mathbf{c}}, B\in \mathbf{c}}g_{C, D}^{B}t_{B}
 \label{eq:Phi1}
\end{align}
is an algebra homomorphism, preserving the identity elements.
\end{thm}

\begin{proof}
Lemma~\ref{lem:P1} is the Property $P_1$ in \cite[\S 1.4]{Lu95} for $\Sc_v$ (together with its canonical basis $\B (\Sc_v)$). Lemma~\ref{lem:P2} shows that $\Sc_v$ and $\B (\Sc_v)$ possess the Property $P_2$ in \cite[\S 1.5]{Lu95}. The Property $P_3$ in \cite[\S 1.7]{Lu95} for $\Sc_v$ and $\B (\Sc_v)$ is just Theorem~\ref{thm:cell structures}(P15).
Hence, by \cite[Proposition 1.9(b)]{Lu95} we have the desired homomorphism $\Phi_{\mathbf{c}}$. Therefore $\Phi$ is also an algebra homomorphism by noting that $\mathbf{J}^{\mathbf S}_{\mathbf{c}}\mathbf{J}^{\mathbf S}_{\mathbf{c}'}=0$ for $\mathbf{c} \neq \mathbf{c}'$.

Next we show that $\Phi$ preserves the identity elements of the two algebras. The identity element of $\Sc_v$ is $\sum_{\nu\in \Lambda_{\ff}}[(\nu,\id,\nu)]$; here we recall $[(\nu,\id,\nu)] =\{(\nu,\id,\nu)\}$. If $D\in \mathcal{D}$ with $\co(D)=\nu$, then by Lemma \ref{lem:d=d1=aaabc} we have $g_{(\nu,\id,\nu), D}^{B}=\delta_{D, B}.$ Hence
\[
\Phi([(\nu,\id,\nu)])=\sum\limits_{D\in \mathcal{D}, \co(D)=\nu}t_{D}.
\]
Consequently, $\Phi(\sum_{\nu\in \Lambda_{\ff}} [(\nu,\id,\nu)])=\sum_{D\in \mathcal{D}}t_{D}.$
\end{proof}
\begin{rem}
By Theorem~ \ref{thm:cell structures} and Corollary~ \ref{preceq-sim-siml}, we can rewrite \eqref{eq:Phi1} as
\begin{align*}
\Phi (\{C\})
&=\sum\limits_{D\in \mathcal{D}, B\in \Xi, D\sim_{L}B}g_{C, D}^{B}t_{B}
\\
&=\sum\limits_{D\in \mathcal{D}, B\in \Xi, \mathfrak{a}(D)=\mathfrak{a}(B)}g_{C, D}^{B}t_{B}.
\end{align*}
\end{rem}
\subsection{A double centralizer property}

Let $\tilde{\mathbf{J}}$ denote the asymptotic Hecke algebra (an associative unital $\mathbb{Z}$-algebra) associated to $\widetilde{\HH}$ (cf. \cite{Lu87a}). By definition, $\tilde{\mathbf{J}}$ is equipped with a $\mathbb{Z}$-basis $\{t_{w}\:|\:w\in \widetilde{W}\}$ with multiplication
\[
t_{x}t_{y}=\sum\limits_{z\in W}\gamma_{x, y}^{z^{-1}}t_{z}.
\]

Suppose $\Lambda_{\ff}$ contains at least one regular $\widetilde{W}$-orbit. Let $\omega$ be a regular orbit in $\Lambda_{\ff}.$ For any two $\gamma, \nu\in \Lambda_{\ff},$ let $\tilde{\mathbf{J}}_{\gamma\nu}$ be the $\mathbb{Z}$-submodule of $\tilde{\mathbf{J}}$ generated by the elements $\{t_{w}\:|\:w\in \mathcal D_{\gamma\nu}^{+}\}$.
We define a tensor space $\mathbb T_{\ff}$ by
\[
\mathbb T_{\ff} =\oplus_{\gamma\in \Lambda_{\ff}}\tilde{\mathbf{J}}_{\gamma\omega}.
\]
Note that $\tilde{\mathbf{J}}$ acts naturally on the right of $\mathbb T_{\ff}$. Set $\mathcal{D}_{\omega}=\{D\in \mathcal{D}\:|\:\co(D)=\omega\}$ and $e=\sum_{D\in \mathcal{D}_{\omega}}t_{D}$. Then $e$ is idempotent and $\mathbf{J}^{\mathbf S}e$ is isomorphic to $\mathbb T_{\ff}.$ Via the isomorphism, $\mathbb T_{\ff}$ becomes a left $\mathbf{J}^{\mathbf S}$-module. We have the following double centralizer property, which can be proved in the same way as for \cite[Theorem 3.2]{Du95}.

\begin{prop}
\label{prop:double-centralizer-property}
Suppose $\Lambda_{\ff}$ contains at least one regular $\widetilde{W}$-orbit. We have
\begin{align*}
\mathbf{J}^{\mathbf S}\cong & \; \End_{\tilde{\mathbf{J}}}(\mathbb T_{\ff}),
\\
& \; \End_{\mathbf{J}^{\mathbf S}}(\mathbb T_{\ff})  \cong  \tilde{\mathbf{J}}.
\end{align*}
\end{prop}

\begin{rem}
The results in Sections~\ref{sec:prelim}--\ref{sec:geom} are affine counterparts of the constructions in \cite{LW22} for $q$-Schur algebras of arbitrary {\em finite} type. On the other hand, the results in Section~\ref{sec:inner} (Theorem~\ref{thm:inner}), Section~\ref{sec:cells} (Proposition~\ref{prop:d=d2}, Theorem~\ref{thm:cell structures}), and Section~\ref{sec:asymp} (Proposition~\ref{characterization-two-cells}, Theorem~\ref{prop:d=d2=a'aa}, Proposition~\ref{prop:double-centralizer-property}) are valid for $q$-Schur algebras of finite type with the same proofs.
\end{rem}

\begin{rem}
According to \cite[Conj. \S3.15]{Lu87c} and \cite[Theorem~ 4]{BFO09}, the asymptotic algebra associated to each cell in a Hecke algebra of finite type admits an equivariant K-group realization. An equivariant K-group realization for the asymptotic algebra associated to each cell in a $q$-Schur algebra of finite type can be formulated similarly.
\end{rem}


\end{document}